\documentclass[a4paper,10pt,reqno,twoside]{amsart}

\usepackage{mathtools,amsmath,amssymb,amsfonts,amsthm,mathrsfs}
\usepackage[pdftex]{graphicx}
\usepackage{graphicx}
\mathtoolsset{showonlyrefs}

\DeclareGraphicsExtensions{.bmp,.png,.pdf,.jpg,}

\usepackage{bm}
\usepackage{verbatim}
\usepackage{color}
\usepackage{mathrsfs}
\usepackage{pgfplots}
\usepackage{euscript}
\usepackage{bbm}
\usepackage{cite}

\newcommand{\limess}[1]%
{

\begin{array}[t]{c}
{\rm ess\, lim}\\
{\scriptstyle #1}
\end{array}

}

\newcommand{\supess}[1]%
{

\begin{array}[t]{c}
{\rm ess\, sup}\\
{\scriptstyle #1}
\end{array}

}
\newcommand{\infess}[1]%
{

\begin{array}[t]{c}
{\rm ess\, inf}\\
{\scriptstyle #1}
\end{array}

}

\def\R{{\mathbb R}}
\def\N{{\mathbb N}}

\def\a{\alpha}
\def\b{\beta}
\def\t{\theta}
\def\d{\delta}
\def\D{\Delta}

\def\s{\sigma}
\def\t{\theta}
\def\l{\lambda}
\def\p{\partial}
\def\O{\Omega}
\def\e{\varepsilon}

\def\o{\omega}
\def\mc{\mathcal}

\def\ms{\mathscr}
\def\ua{\uparrow}
\def\da{\downarrow}

\numberwithin{equation}{section}

\theoremstyle{definition}

\theoremstyle{definition}
\newtheorem{remark}{Remark}[section]

\theoremstyle{plain}
\newtheorem{theorem}{Theorem}[section]
\newtheorem{proposition}{Proposition}[section]
\newtheorem{lemma}{Lemma}[section]

\begin{document}

\title[Non-negative solutions of a sublinear elliptic problem]
{Non-negative solutions of a sublinear elliptic problem
\vskip1cm}

\thanks{The authors have been  supported by the Research Grant PID2021-123343NB-I00 of the  Ministry of Science and Innovation of Spain, and the Institute of Interdisciplinary Mathematics of Complutense University of Madrid.}

\author{Juli\'{a}n L\'{o}pez-G\'{o}mez}
\address{Juli\'{a}n L\'{o}pez-G\'{o}mez:
Instituto Interdisciplinar de Matem\'aticas,
Universidad Complutense de Madrid,
Madrid, Spain}
\email{jlopezgo@ucm.es}
\author{Paul H. Rabinowitz}
\address{Paul H. Rabinowitz: Department of Mathematics, University of Wisconsin-Madison, Madison, WI53706, USA}
\email{rabinowi@math.wisc.edu}
\author{Fabio Zanolin}
\address{Fabio Zanolin: Dipartimento di Scienze Matematiche, Informatiche e Fisiche, Universit\`a degli Studi
di Udine, Via delle Scienze 2016, 33100 Udine, Italy}
\email{fabio.zanolin@uniud.it}

\begin{abstract}
\vskip.5cm
In this paper the existence of solutions, $(\lambda,u)$,
of the problem
$$\left\lbrace\begin{array}{ll}  -\D u=\l u -a(x)|u|^{p-1}u & \quad \hbox{in }\O,\\
  u=0 &\quad \hbox{on}\;\;\p\O,
  \end{array}\right.
$$
is explored for  $0 < p < 1$. When $p>1$, it is known that there
is an unbounded component of such solutions bifurcating from
$(\s_1, 0)$, where $\s_1$ is the smallest eigenvalue of $-\D$ in
$\O$ under Dirichlet boundary conditions on $\p\O$. These
solutions have $u \in P$, the interior of the positive cone. The
continuation argument used
when $p>1$ to keep $u \in P$
fails if $0 < p < 1$. Nevertheless when $0 < p < 1$,
we are still able to show that there is a component of
solutions bifurcating from $(\s_1, \infty)$, unbounded outside of
a neighborhood of $(\s_1, \infty)$, and having $u \gneq 0$. This
non-negativity for $u$ cannot be improved as is shown via a
detailed analysis of the simplest autonomous one-dimensional
version of the problem: its set of non-negative solutions
possesses a countable set of components, each of them consisting
of positive solutions with a fixed (arbitrary) number of bumps.
Finally, the structure of these components is fully described.
\vskip.5cm
\end{abstract}

\subjclass[2010]{35B09,35B25,35B32,35J15}

\keywords{Non-negative solutions. Sublinear elliptic problems. Bifurcation from infinity. Singular perturbations. Non-negative multi-bump solutions. Global structure.}
\vspace{0.1cm}

\date{\today}

\maketitle

\setcounter{page}{1}
\section{Introduction}

\noindent Consider the sublinear  elliptic boundary value problem
\begin{equation}
\label{1.1}
  \left\lbrace\begin{array}{ll}  -\D u=\l u -a(x)|u|^{p-1}u & \quad \hbox{in }\O,\\
  u=0 &\quad \hbox{on}\;\;\p\O,
  \end{array}\right.
\end{equation}
where $p\in (0,1)$, $\O$ is a bounded domain of  $\R^N$, $N\geq 1$, having a smooth boundary,   $\p\O$. The weight function $a:\bar\O\to\R$ is H\"{o}lder continuous and satisfies the following condition:
\begin{equation}
\label{1.2}
   a(x)>0\;\; \hbox{for all}\;\; x\in\bar \O.
\end{equation}
A classical solution of \eqref{1.1} is a pair $(\lambda, u) \in \R \times C^2(\bar\Omega).$
Let
$$
  \mathcal{C}_0^1(\bar\O)=\left\{u\in\mathcal{C}^1(\bar\O)\;:\; u|_{\p\O}=0\right\}
$$
and
$$
P = \{u \in  \mathcal{C}_0^1(\bar\O) : u > 0\;\; \hbox{in}\;\; \Omega \;\;\hbox{and} \;\; \tfrac{\p u}{\p n}<0 \;\; \hbox{on} \;\; \p\O\}.
$$
Here $n= n(x)$ is the outward unit normal to $\O$ at $x\in\p\O$. Thus, $P$ is the interior of the positive cone in $\mathcal{C}_0^1(\bar\O)$.
When $u\in P$, we write $u\gg 0$.
\par
When $p > 1$, it is known \cite[Sect. 7.1]{LG01} that \eqref{1.1} has an unbounded component of solutions that bifurcates from $(\s_1, 0)$ where $\s_1$
is the smallest eigenvalue of
\begin{equation}
\label{1.3}
  \left\lbrace\begin{array}{ll}  -\D u=\l v & \quad \hbox{in }\O,\\
  v=0 &\quad \hbox{on}\;\;\p\O,
  \end{array}\right.
\end{equation}
i.e. of the linearization of (1.1) about $u=0.$
Moreover, aside from $(\s_1, 0)$, for any $(\lambda, u)$ on this component, $u\in P$  and the projection of this component on $\R$ is $(\s_1, \infty)$. These results are a consequence of $p>1$,
the global bifurcation theorem, the properties of the smallest eigenpair of \eqref{1.3},
and a continuation argument based on the maximum principle  that forces $u$ to remain in $P$. When $N=1$, the continuation argument follows more simply from the fact that if $u$ is any solution of the associated equation that has a double zero, i.e. $u(y) = u'(y) = 0$, then $u \equiv 0$.
\par
One of the two main goals of this paper is explore to what extent
the result just mentioned for  $p>1$ carries over to the case of
$0 < p < 1$.  Towards that end, observe first that the formal
linearization of the operator in \eqref{1.1} about $u=0$ yields a
singular linear operator that does not permit bifurcation. So, the
previous existence argument does not work. However as will be
shown in detail in Section 2, due to the sublinear nonlinearity,
one can invoke a variant of the global bifurcation theorem to get
bifurcation from $\infty$ here. Indeed such results are already
known for equations related to \eqref{1.1} (see, e.g.,
\cite{Rab71,Rab73}). For the current setting, they tell us there
is a component, $\mc{K}$, of solutions of  \eqref{1.1} that is
unbounded outside of a neighborhood of $(\s_1, \infty)$. An
additional argument shows that the solutions $(\lambda, u)$ on
$\mc{K}$ near $(\s_1, \infty)$  have $u \in P$. However
unfortunately due to $0 < p < 1$, the maximum principle argument
used earlier to continue solutions within $P$ fails. Thus, it is
possible that the projection of $\mc{K}$ on
$\mathcal{C}_0^1(\bar\O)$ leaves $P$. It will be shown that
somewhat surprisingly, $\mc{K}$ has a subcomponent lying in
$\R\times (P\setminus\{0\})$, unbounded outside of a neighborhood
of $(\s_1, \infty)$, and whose $\l$-projection on $\R$ is the
interval $(\s_1, \infty)$. To obtain this result, a more indirect
approach to study $\mc{K}$ will be taken. For $\e > 0$, consider
the family of equations
\begin{equation}
\label{1.4}
  \left\lbrace\begin{array}{ll}  -\D u=\l u -a(x) g_\e(u) u & \quad \hbox{in }\O,\\
  u=0 &\quad \hbox{on}\;\;\p\O,
  \end{array}\right.
\end{equation}
where the nonlinearity $g_\e$ is chosen to satisfy
\begin{equation*}
  g_\e(u):= (\e+u)^{p-1}\quad \hbox{if}\;\; u \geq 0
  \end{equation*}
 and $g_{\e}$ is even. It will be shown that for each such $\e$, there is a connected set of solutions of \eqref{1.4} joining $(\s_1, \infty)$ to $(\lambda_1(\e), 0)$ lying in $\R \times P$. Here
 $\lambda_1(\e)$ is the smallest eigenvalue of the linearization of \eqref{1.4} about $u=0$. These new continua are approximations to part of $\mc{K}$ since letting $\e \da 0$, a further  argument yields the component of solutions in $\R \times \bar P$ just mentioned.
\par
 Our second main goal is to understand the nature of solutions of \eqref{1.1} in $(\R \times \bar P) \setminus(\R \times P)$. This is a difficult question in the generality of \eqref{1.1}. However in Section 3, an exhaustive  study will be made of the simplest one-dimensional version of \eqref{1.1}:
\begin{equation}
\label{1.5}
  \left\lbrace\begin{array}{l}  -u''=\l u -a |u|^{p-1}u  \quad \hbox{in }[0,L],\\
  u(0)=u(L)=0,
  \end{array}\right.
\end{equation}
where $a>0$ is a positive constant and $L>0$ is given.  A key role here is played by the fact that there can be nonuniqueness for solutions of the initial value problem:
\begin{equation}
\label{1.6}
  \left\lbrace\begin{array}{l}  -u''=\l u -a |u|^{p-1}u, \\
  u(0)=u_0, \;\; u'(0)=v_0, \end{array}\right.
\end{equation}
when $u_0 = 0$ since the nonlinearity does not satisfy a Lipschitz
condition at that point. Thus it is possible that \eqref{1.6}
possesses a solution such that $u(0) = 0 = u(L) = u'(0) = u'(L),
u(x) = 0$ for $x \in (0, x_1) \cup (x_2, L)$ where $0 \leq  x_1 <
x_2  \leq L$, and $u(x) > 0$ for $x \in (x_1, x_2)$. Such a
homoclinic solution of \eqref{1.6} will be referred to as a one
bump solution of the equation. The concatenation of $j$ such
solutions will be called a $j$-bump solution of \eqref{1.6}.
\par
A classification will be given of the set of non-negative solutions of \eqref{1.5}. It turns out that these solutions form a countable set of disjoint components, $\ms{D}_j^+$, $j\geq 1$, possessing the following properties:
\begin{itemize}
\item For every $j\geq 1$ and $(\l,u)\in \ms{D}_j^+$, $u$ has
exactly $j$ bumps in the interval $(0,L)$, the bumps occurring at
some points $x_i$, $1\leq i\leq j$, at which
$$
   u(x_i)\geq \left(\tfrac{2a}{(p+1)\l}\right)^\frac{1}{1-p}.
$$

\item
$\mc{P}_\l(\ms{D}_1^+)=(\s_1,\infty)$, and
$\mc{P}_\l(\ms{D}_j^+)=[\Sigma_j,\infty)$ for all $j\geq 2$, where
$\mc{P}_\l$ stands for the $\l$-projection operator, $\mc{P}_\l
(\l,u)=\l$,  and
$$
  \Sigma_j:= \left(\tfrac{2}{1-p}\right)^2 \sigma_j \quad \hbox{for all}\;\; j\geq 1, \quad \hbox{with}\;\;
  \s_j := \left( \frac{j\pi}{L}\right)^2\quad \hbox{for all}\;\; j\geq 1.
$$
\item $\ms{D}_i^+\cap \ms{D}_j^+=\emptyset$ if $i\neq j$. \item In
the interval $(\s_1,\Sigma_1]$, $\ms{D}_1^+$ is a curve,
$(\l,u_\l)$, such that $\lim_{\l \da \s_1}\|u_\l\|_\infty=\infty$.
Moreover, for every fixed $\l>\Sigma_1$, $\ms{D}_1^+$ is a
one-dimensional simplex, i.e. a line segment, which expands as
$\l$ increases.
\item For every $j\geq 2$ and $\l>\Sigma_j$,
$\ms{D}_j^+$ is a $j$--dimensional simplex, which expands as $\l$
increases.
\end{itemize}
These properties are sketched In Figure \ref{fig310b} in Section 3.
They contrast very strongly  with the behavior of the
superlinear model, with $p>1$, where the set of positive solutions
of \eqref{1.5} consists of a curve, $(\l,u_\l)$,
where $u_\l$ increases as $\l$ increases and is such that
$\lim_{\l\da \s_1}\|u_\l\|_\infty=0$.
\par
Some other work on sublinear problems will be mentioned next. In \cite{LMZ},
a very preliminary study  has been carried out recently.
The existence of a (unique) solution with $u >0$ in $(0,L)$ and $u'(0)>0$
was established for every $\l<\pi^2$
when $a$ is a negative constant.
Thus condition \eqref{1.2} is not satisfied.
Similarly, \cite{PT} analyzed the Cauchy problem
\begin{equation}
\label{1.7}
  \left\{ \begin{array}{ll} u''=a(x)u^p  & \quad \hbox{in}\;\; [0,\infty),\\
  u(0)=u_0>0,\;\; u'(0)=0,\end{array}\right.
\end{equation}
where $p\in (0,1)$ and $a \in\mc{C}^1(\R)$ is an even function, increasing in $[0,\infty)$, such that,
for some $\a>0$,
\begin{equation}
\label{1.8}
  (x-\a)a(x)\geq 0\;\;\hbox{for all}\;\; x\in [0,\infty).
\end{equation}
Although $p\in (0,1)$, the fact that in \cite{PT} the
initial condition $u_0$ is positive implies \eqref{1.7}
has a unique solution for each such $u_0$,
thereby avoiding the nonuniqueness phenomena we encounter.
According to \cite[Th. 1.1]{PT}, there exists $u_0>0$ such that
$$
  u(\xi)=u'(\xi)=0\;\;\hbox{for some}\;\; \xi:=\xi(u_0)\in\R.
$$
This suggests that the problem \eqref{1.6} for the choice $u_0=0=v_0$
might have non-zero solutions.
However no analysis of the structure of such solutions was carried out in \cite{PT}.
Note that \eqref{1.8} cannot be satisfied if $a(x)$ is a positive constant,
which is the case dealt with here in Section 3.
Some further developments were given in \cite{BPT},
where the non-negative solutions of PDE's with a sign-indefinite weight were analyzed.
\par
In another context, the second order periodic sublinear problem
\begin{equation}
\label{1.9}
  \left\{ \begin{array}{ll} u''=a(x)u^p-b(x)u^q   & \quad  x\in [0,L], \\
  u(0)=u(L), \;\; u'(0)=u'(L),\end{array}\right.
\end{equation}
where $L>0$, $a, b \in\mc{C}(\R/L,\R)$ and $0<p<q<1$,
was analyzed in \cite{CS} as a model for studying
a valveless pumping effect in a simple pipe-tank configuration.
This model has some applications to the study of blood circulation
(see \cite[Ch. 8]{To} for further details).
In particular, for positive $a(x)$ and $b(x)$
the main result of \cite{CS} establishes the solvability of \eqref{1.9}.
Indeed, the Mountain Pass Theorem \cite{AR} shows that \eqref{1.9}
has a non-trivial non-negative periodic solution.
But the problem of ascertaining the structure of these
non-trivial solutions was not addressed in \cite{CS}.
\par
Some results have been recently given in \cite{KRUJ} and \cite{KRUC}
for the multidimensional problem:
\begin{equation}
\label{1.10}
  \left\{ \begin{array}{ll} -\D u =a(x)u^p   & \quad \hbox{in} \;\; \O, \\
   u =0 & \quad \hbox{on}\;\; \p\O,\end{array}\right.
\end{equation}
where $\O$ is a bounded and smooth domain in $\R^N$, $N\geq 1$.
The function $a\in L^r(\O)$ (for $r>N$) and changes sign in $\O$.
It was shown in \cite{KRUJ} that there exists $\d>0$ such that
any nontrivial nonnegative solution of \eqref{1.10}
lies in the interior of the positive cone of $\mc{C}_0^1(\bar \O)$ if
\begin{equation}
\label{1.11}
  \|a^-\|_{L^r(\O)}\leq \d,
\end{equation}
where $a^-$ stands for the negative part of $a(x)$.
This result provides us with the
strong positivity of all positive solutions of \eqref{1.1} for $\l=0$ when $a(x)$ is positive everywhere, However this result cannot be applied to \eqref{1.1} under condition \eqref{1.2} because  in our setting, $a(x)<0$ for all $x\in \O$ and we are not imposing any restriction on the size of $a(x)$.  Thus although the existence of positive solutions
vanishing somewhere was obtained in \cite{KRUC}, our problem \eqref{1.1} remains outside the
scope of \cite{KRUJ} and \cite{KRUC}.
\par
The remainder of this paper is arranged as follows. Section 2 presents the proof
of the multidimensional existence theorem.
Section 3 provides a complete classification of
the set of non-negative solutions of the one-dimensional problem \eqref{1.5}.
A challenging open question is to provide an analogue this classification
for a general function $a(x)$.
The analysis of Section 3 suggests
the set of non-negative solutions of \eqref{1.1}
may be very complex when $N>1$ and
and further obtaining results in that direction remains an even more difficult
open question.

\section{The existence of non-negative solutions for \eqref{1.1}}\label{section-2}

\noindent In this section, our main result on the existence of non-negative solutions of \eqref{1.1} will be proved. For every smooth subdomain $D$ of $\O$ and sufficiently smooth function, $V$, on $D$, $\s_1[-\D+V,D]$ denotes the smallest eigenvalue of $-\D+V$ in $D$ under Dirichlet boundary conditions on $\p D$. Through the rest of this section, we set
\begin{equation*}
  \s_1\equiv \s_1[-\D,\O].
\end{equation*}

Before stating the main result of this section, the existence of a component of solutions of \eqref{1.1} that bifurcate from $(\sigma_1, \infty)$ will be discussed. To set the stage, following \cite{Rab73} and setting
$$
\|u\| \equiv \|u\|_{\mc{C}^1(\bar\Omega)},
$$
the  change of variables
\begin{equation}
\label{2.2}
  w:= \frac{u}{\|u\|^2}
\end{equation}
interchanges the roles of $0$ and $\infty$ and converts questions of bifurcation from $\infty$
to bifurcation from $0$ and conversely. In particular it transforms \eqref{1.1} into
\begin{equation}
\label{2.3}
    \left\lbrace\begin{array}{ll}  -\D w=\l w -a(x) \|w\|^{2(1-p)}|w|^{p-1}w & \quad \hbox{in }\O,\\
  w=0 &\quad \hbox{on}\;\;\p\O,
  \end{array}\right.
\end{equation}
whose nonlinear term,
\begin{equation*}
  g(w):=\|w\|^{2(1-p)}|w|^{p-1}w,\qquad w\in \mc{C}^1(\bar\O),
\end{equation*}
satisfies
$$
  \|g(w)\|= \|w\|^{2-p}=o(\|w\|) \quad \hbox{as}\;\; \|w\|\to 0,
$$
because $2-p>1$. Note that \eqref{2.2} implies
$$
  \|w\|=\frac{1}{\|u\|}.
$$
Thus, information about solutions of \eqref{1.1} can be obtained from
solutions of the superlinear equation \eqref{2.3}.
For what follows, $\mc{C}^2_0(\bar\O)$ stands for the closed subspace of $\mc{C}^2(\bar\O)$ consisting of all functions $w\in \mc{C}^2(\bar\O)$ vanishing on $\p\O$. Let $w \in \mc{C}^1(\bar\O)$. Then by the Schauder linear elliptic existence and regularity theories, there is a unique $T(\l,w) \in \mc{C}^2_0(\bar\O)$ satisfying
\begin{equation}
\label{2.4}
   \left\lbrace\begin{array}{ll} -\D T(\l,w)=\l w -a(x) \|w\|^{2(1-p)}|w|^{p-1}w & \quad \hbox{in}\;\; \O,\\  T(\l,w)=0, &\quad \hbox{on}\;\;\p\O.
  \end{array}\right.
\end{equation}
Finding a solution of \eqref{2.3} (for fixed $\l$) is equivalent to finding a fixed point of the operator $T(\l,w)$.
Note that $T(\l,\cdot): \mc{C}^1(\bar\O)\to \mc{C}^1_0(\bar\O)$ is compact and $\sigma_1$ is a simple eigenvalue of the linearization of \eqref{2.3} about $w=0$. Therefore, applying Theorems 1.6, 2.28 and Corollary 1.8 of \cite{Rab73}, there exists a component of the set of nontrivial solutions of \eqref{1.1}, $\mathscr{C}$, bifurcating from $u=\infty$ at $\l=\s_1$. Moreover, there is a neighborhood, $O$, of $(\s, \infty)$ in which $\mathscr{C}$ decomposes into a pair of subcontinua, $\mathscr{C^+}$ and $\mathscr{C^-}$ such that
$$
O \cap \mathscr{C} = \mathscr{C^+} \cup \mathscr{C^-},\;\;
\mathscr{C^+} \cap \mathscr{C^-} = \{(\s_1, \infty)\}\;\; \hbox{and} \;\; \mathscr{C}^{\pm}
\subset  \{(\sigma_1, \infty)\} \times (\pm P).
$$
Remember that $P$ stands for the interior of the positive cone in $\mathcal{C}_0^1(\bar\O)$.
\par
Let $\mc{P}_\l$ denote the projector of $\R \times \mc{C}^1(\bar \Omega)$ to $\R$.
It is straightforward to show that $\mc{P}_\l \mathscr{C}^+ $ is contained in the interval $(\s_1,\infty)$ and that \eqref{1.1} cannot admit a positive solution if $\l\leq \s_1$.
While points $(\l, u)$ on $\mathscr{C}^+$ near $(\s_1, \infty)$ have $u \in P$, $\mathscr{C}^+$ may exit $\R \times P$ outside of the set $O$.
However somewhat surprisingly,  our main result in this section says in general there is a large component of solutions of \eqref{1.1} in $\R \times (\bar P\setminus\{0\})$, namely:

\begin{theorem}
\label{th2.1} If \eqref{1.2} holds, there is a component, $\ms{D}^+$, of the set of  nonnegative solutions of \eqref{1.1} with $(\s_1,\infty)\in \bar{\ms{D}}^+ \subset \R \times (\bar P\setminus\{0\})$. Moreover, $\mc{P}_\l(\ms{D}^+)=  (\s_1,\infty)$, and there is a neighborhood, $\hat O$, of $(\s_1, \infty)$ such that $(\l, u) \in \hat O \setminus \{(\s_1, \infty)\}$ implies $u \in P$, i.e. $u\gg 0$.
\end{theorem}

 It should be stressed that
$$
  \bar P =\{u\in\mc{C}_0^1(\bar\O)\;:\; u\geq 0\;\;\hbox{in}\;\;\bar\O\;\;
 \hbox{and} \;\; \tfrac{\p u}{\p n}\leq 0 \;\; \hbox{on} \;\; \p\O \},
$$
and that $P=\mathrm{int\,}\bar P$.

\begin{remark}
\label{re2.1} Theorem \ref{th2.1} still holds, with an almost identical proof, if, instead of \eqref{1.2}, we assume that $a\gneq 0$ with $a^{-1}(0)=\bar\O_0\subset \O$ for some smooth $\O_0$, though in this case we can only guarantee that $\mc{P}_\l(\ms{D}^+)\subset (\s_1,\s_0)$,
where $\s_0$ = $\s_1[-\D,\O_0]$. In this case, multiplying the differential equation by the associated principal eigenfunction and integrating by parts shows that $\l\in (\s_1,\s_0)$ if \eqref{1.1} has a positive solution. It remains an open problem to ascertain whether or not, $\mc{P}_\l(\ms{D}^+)= (\s_1,\s_0)$ in the degenerate case. This degenerate
problem will be dealt with elsewhere.
\end{remark}

Our proof of Theorem \ref{th2.1} requires a rather different approach than was taken to obtain $\mathscr{C}$. Equation \eqref{1.1} will be approximated by a family of related equations for which one has bifurcation from both $0$ and $\infty$ and Theorem \ref{th2.1} will be obtained via a limit process using the properties of the new family of equations.
Accordingly, for $\e > 0$, consider the family of equations
\begin{equation}
\label{2.5}
 \left\lbrace\begin{array}{ll}
 -\D u=\l u -a(x) g_\e(u) u & \quad \hbox{in }\O,\\
  u=0 &\quad \hbox{on}\;\;\p\O,
  \end{array}\right.
\end{equation}
where the function $g_\e$ satisfies
\begin{equation*}
  g_\e(u):= (\e+u)^{p-1}\quad \hbox{if}\;\; u \geq 0
  \end{equation*}
  and $g_\e$ is even.
 Because of the form of $g_\e$,

$$
  g_\e(z)=\e^{p-1}+(p-1)\e^{p-2}z+O(z^2) \qquad \hbox{as}\;\; z\to 0.
$$
Define $\l_1(\e)$ by
\begin{equation}
\label{2.6}
  \l_1(\e):= \s_1\Big[-\D +\frac{a(x)}{\e^{1-p}},\O\Big],\qquad \e>0,
\end{equation}
i.e.  $\l_1(\e)$ is the smallest eigenvalue of the linearization of \eqref{2.5}  about $u = 0$. Similarly, $\s_1=\s_1[-\D,\O]$.  Equation \eqref{2.6} implies that $\lambda_1(\e) \to \infty$ as $\e \to 0$
(see item (b) in the next proposition). Thus for convenience in what follows, it can be assumed that $\e$ is small enough so that
\begin{equation*}
 \s_1 < \l_1(\e).
\end{equation*}
Corresponding to \eqref{2.5} is its transformed version under the change of variables \eqref{2.2}:
 \begin{equation}
\label{2.7}
\left\lbrace\begin{array}{ll}
 -\D w=\l w -a(x) \frac{1}{(\e +w/ \|w\|^2)^{1-p}} w & \quad \hbox{in }\O,\\
  w=0 &\quad \hbox{on}\;\;\p\O.
  \end{array}\right.
\end{equation}
As in \eqref{2.4}, we use \eqref{2.7} to define $T_\e(\l,w)$ via
\begin{equation*}
\left\lbrace\begin{array}{ll}
-\D T_\e(\l,w)=\l w -a(x) \frac{1}{(\e +w/ \|w\|^2)^{1-p}} w & \quad \hbox{in }\O,\\
  T_\e(\l,w)=0 &\quad \hbox{on}\;\;\p\O.
  \end{array}\right.
\end{equation*}
Then, the compact operator form of \eqref{2.7} is
\begin{equation}
\label{2.8}
w = T_\e(\l,w).
\end{equation}
The next result studies some properties of the solutions of \eqref{2.5}.
Recall that $\l_1(\e)$ is defined in \eqref{2.6}.
\begin{proposition}
\label{pr2.2}
Suppose $a(x)>0$ for all $x\in\bar \O$. Then:
 \begin{enumerate}
\item[{\rm (a)}] For each $\e > 0$, there is a component of solutions, $\ms{D}_\e^+$, of \eqref{2.5} joining $(\s_1,\infty)$ and $(\lambda_1(\e), 0)$. Moreover aside from these endpoints, $(\l, u) \in {\ms{D}_{\e}^+}$ implies  $\l \in   (\s_1,\lambda_1(\e))$ and $u \in  P$.
\item[{\rm (b)}]$\lambda_1(\e) \to \infty$ as $\e \to 0$.
\item[{\rm (c)}] For each $\e > 0$, there is a component of solutions, $\ms{E}_\e^+$, of \eqref{2.7} joining $(\s_1,0)$ and $(\lambda_1(\e), \infty)$. Moreover aside from these endpoints, $(\l, u) \in {\ms{E}_{\e}^+}$ implies  $\l \in   (\s_1,\lambda_1(\e))$ and $u \in  P$.
\item[{\rm (d)}]   For each $\l > \s_1$, there is an $r(\l) >0$ that is independent of $\e$ such that whenever $(\l, w) \in \ms{E}_{\e}^+ \setminus\{(\s_1,0) \cup (\lambda_1(\e), \infty)\}$,  then
$\|w\|_{C^1(\bar \Omega)}>r(\l).$
\end{enumerate}
\end{proposition}
\begin{proof}
By Theorem 2.12 of \cite{Rab71}, \eqref{2.5} has a component, $\mathscr{D}_\e$, of solutions lying in $\R \times P$ that bifurcates from $(\l_1(\e), 0)$ and is unbounded (see also \cite[Ch. 6]{LG01} for more details). Rewriting \eqref{2.5} as
\begin{equation}
\label{2.9}
  \left\lbrace\begin{array}{ll}
 (-\D +\frac{a(x)}{(\e +u)^{1-p}}) u=\l u & \quad \hbox{in }\O,\\
  u=0 &\quad \hbox{on}\;\;\p\O,
  \end{array}\right.
\end{equation}
and noting that $u \in P$ shows that $\l$ in \eqref{2.9} satisfies
\begin{equation}
\label{2.10}
\l = \s_1\Big[-\D + \frac{a(x)}{(\e +u)^{1-p}}, \O\Big].
\end{equation}
Since $\s_1[-\D+V,\O]$ depends monotonically on $V$ (see, e.g., \cite{LG96,LG13}), it follows from \eqref{1.2} that
$$
\s_1[-\D +\frac{a(x)}{\e^{1-p}},\O] = \l_1(\e) > \s_1[-\D + \frac{a(x)}{(\e +u)^{1-p}}, \O] = \l >
\s_1[-\D, \O] = \s_1.
$$
Thus, $\l\in (\s_1,\l_1(\e))$. In particular,
$$
  \mc{P}_\l ( \ms{D}_\e)\subset [\s_1,\l_1(\e)].
$$
It remains to show that $\mc{P}_\l ( \ms{D}_\e)= [\s_1,\l_1(\e)]$. The definition of $\mathscr{D}_\e$ shows that $\l_1(\e)$ must be the right endpoint of $\mc{P}_\l ( \ms{D}_\e)$. Since $\ms{D}_\e$ is unbounded, there is a sequence of points on it converging to a bifurcation point at infinity. But the only such point that can be the limit of members of $\R \times P$ is $(\s_1, \infty)$. Thus $(a)$ is proved.
\par
For $(b)$, note that, again by the monotonicity property of $\s_1[-\D+V,\O]$ with respect to $V$,
$$
\l_1(\e) > \s_1[-\D +\frac{a_L}{\e^{1-p}},\O] = \frac{a_L}{\e^{1-p}}+\s_1[-\D, \O] \to \infty \;\;
\hbox{as}\; \; \e \da 0,
$$
where
$$
   a_L:= \min_{\bar \O}a>0.
$$
The next item, $(c)$, is an immediate consequence of $(a)$ and the inversion \eqref{2.2}.
To prove $(d)$, suppose it is false. Then there is a $\l > \s_1$, a sequence, $\e_i \to 0$ as $i \to \infty$, and a sequence $(\l, w_i) \in \ms{E}_{\e_i}$ with
$$
   \lim_{i\to \infty} \|w_i\|_{\mc{C}^1(\bar \O)}= 0.
$$
Setting
$$
    v_i := \frac{w_i}{\|w_i\|},\qquad i\geq 1,
$$
rewriting \eqref{2.7} for $\e= \e_i$ and $w = w_i$, and dividing by $\|w_i\|$, yields:
\begin{equation}
\label{2.13}
\left\lbrace\begin{array}{ll}
 -\D v_i=\l v_i -a(x) \frac{1}{(\e_i +w_i/ \|w_i\|^2)^{1-p}} v_i & \quad \hbox{in }\O,\\
  v_i=0 &\quad \hbox{on}\;\;\p\O.
  \end{array}\right.
\end{equation}
As $i \to \infty$, the nonlinear term in \eqref{2.13} goes to $0$ while due to the compactness of
the inverse of $-\D, \;v_i$ converges in $\mc{C}^2(\bar \O)$ to a solution, $v$, of
\begin{equation}
\label{2.14}
\left\lbrace\begin{array}{ll}
 -\D v=\l v  & \quad \hbox{in }\O,\\
  v=0 &\quad \hbox{on}\;\;\p\O,
  \end{array}\right.
\end{equation}
where $v \geq 0$ in $\O$ and $\|v\| =1$. It follows that $v$ is the unique positive
eigenfunction of norm $1$ of \eqref{2.14}, the entire sequence converges to $v$ and $\l = \s_1$. But this contradicts the hypothesis that $\l>\s_1$. So, item (d) follows.
\end{proof}

\begin{remark}
\label{rm2.2}  Part (d) of Proposition \ref{pr2.2} readily implies that, for any subinterval
$[\alpha, \beta]$ of $(\s_1, \l_1(\e))$, there is an $R(\alpha, \beta)>0$ independent of $\e$ such that
$$
   \|w\|_{\mc{C}^1(\bar \Omega)}>R(\alpha, \beta)
$$
whenever $\l \in [\alpha, \beta]$ and $(\l, w) \in \ms{E}_{\e^+} \setminus\{(\s_1,0) \cup (\lambda_1(\e), \infty)\}$.
\end{remark}

Now we are ready for the
\par
\vspace{0.2cm}

\noindent\textbf{Proof of Theorem \ref{th2.1}:}
Let $B$ be a bounded open neighborhood of $(\s_1,0)$ in $\mc{C}^1(\bar \O)$. Then, for each $\e > 0$,
by Parts (b) and  (c) of Proposition \ref{pr2.2},
\begin{equation}
\label{2.15}
    \ms{B}_\e \equiv \p B \cap \ms{E}_{\e} \cap ((\s_1, \infty) \times P)) \neq \emptyset.
\end{equation}
Let $(\e_i)$ be a decreasing sequence such that $\e_i \to 0$ as $i \to \infty$, and choose any sequence $(\l_i, w_i) \in \ms{B}_{\e_i}$. According to \eqref{2.15}, the functions $w_i$ are bounded in $\mc{C}^1(\bar \O)$ with $w_i \in P$. Thus they are classical solutions of \eqref{2.7}  and its associated compact operator equation \eqref{2.8}. Rewriting \eqref{2.8} for $\e =\e_i$, $\l= \l_i$ and $w = w_i$, as
\begin{equation}
\label{2.16}
\left\lbrace\begin{array}{ll}
-\D T_{\e_i}(\l_i,w_i)= \l_i w_i-a(x) (\frac{w_i}{\e_i + w_i /\|w_i\|^2})^{1-p} w_i^{p} & \quad \hbox{in }\O,\\  T_{\e_i}(\l_i,w_i)=0 &\quad \hbox{on}\;\;\p\O,
  \end{array}\right.
  \end{equation}
the Schauder linear existence and regularity theories show the functions, $w_i$, are bounded in $\mc{C}^{2, \gamma}(\bar \O)$ for some $\gamma > 0$. Therefore, since the right hand side of the differential equation in \eqref{2.16} is bounded in $\mc{C}(\bar \O)$, along a subsequence, $(\l_i, w_i)$ converges to a
$$
   (\l^*, w^*) \in \p B \cap( [\s_1, \infty)) \times \bar P))
$$
with the convergence in $w_i$ being in $\mc{C}^2(\bar \O)$. By Part (d) of Proposition \ref{pr2.2} and  Remark \ref{rm2.2}, $w^* \not \equiv 0$ unless $\l^* = \s_1$. But then $(\s_1, 0) \in \p B$, contrary to the choice of $B$.
\par
Next we claim that $w^*$ is  a solution of \eqref{2.5} for $\l=\l^*$. Indeed, let
$$
A(w^*) := \{x \in \O\;:\; w^*(x) > 0\}.
$$
Since $w^*$ is continuous, $A(w^*)$ is open. Let $z \in A(w^*)$. Then,
from \eqref{2.7} with $\e =\e_i$, $\l= \l_i$ and $w = w_i$, we see $(\l^*, w^*)$ satisfies \eqref{2.3} in a neighborhood of $z$ in $A(w^*)$. Hence, $(\l^*, w^*)$ satisfies \eqref{2.3} in all of $A(w^*)$. But
$w^* \in \mc{C}^2(\bar \O)$. Thus, by continuity, $-\D w^*= 0$ on $\p A(w^*)$ and $w^* \equiv 0$ in
$\O \setminus \overline{A(w^*)}$. Therefore $(\l^*, w^*)$ satisfies \eqref{2.3} as well on these sets. Lastly, $w^*= 0$ on $\p \O$. Consequently $(\l^*, w^*)$ satisfies \eqref{2.3}.
\par
The argument just employed applies to any $(\l, w)$ that is obtained by the same limit process. Let
$\ms{S}$ denote the closure of the set of such solutions of \eqref{2.3}
in $\R \times \mc{C}^1(\bar \O)$.
Note that due to Part (d) of Proposition \ref{pr2.2} and Remark \ref{rm2.2}, $\ms{S}$ contains no "trivial solutions" $(\l,0)$ of \eqref{2.3} other than $(\s_1, 0)$. However, it might contain solutions $(\l,w)$ with $w\in \p P \setminus\{0\}$ that are limits of solutions $(\l_i,w_i)$ with $w_i\in P$.
\par
To complete the proof of Theorem \ref{th2.1}, it must be shown that $\ms{S}$ contains a component whose $\l$-projection on $\R$ is $[\s_1, \infty)$. Then, the transformation \eqref{2.2} applied to $\ms{S}$ provides $\ms{D}^+$ of Theorem \ref{th2.1} and the proof is complete. If Theorem \ref{th2.1} is false, the component, $\ms{S}_0$, of $\ms{S}$ to which $(\s_1,0)$ belongs projects on $\R$  to a bounded interval $[\s_1, b]$, for some $b>\s_1$, the lower endpoint being $\s_1$ via our construction. We claim there is an $M = M(b) > 0$ such that
\begin{equation}
\label{2.17}
  (\l, w) \in \ms{S}_0\;\; \hbox{implies}\;\; \|w\| < M.
\end{equation}
To see this, note that, by construction,  any $(\l,w) \in \ms{S}_0$ is a $\mc{C}^2(\bar \O)$--limit of solutions, $(\l_j, w_j)$, $j\geq 1$, of \eqref{2.7} with $\e = \e_j$. Without loss of generality, it can be assumed that $\l_j < b+1$ for all $j \in \N$. On the other hand, by \eqref{2.10},
\begin{align*}
\l_j  & = \s_1[-\D + \frac{a(x)}{(\e_j +w_j/\|w_j\|^2)^{1-p}}, \O], \qquad j\geq 1.
\end{align*}
Thus, since $a\geq a_L$ and
$$
   \frac{w_j}{\|w_j\|^2}\leq \|w_j\|^{-1},\qquad j\geq 1,
$$
it follows from the monotonicity of the principal eigenvalue with respect to the potential that
\begin{equation}
\label{2.18}
 \l_j \geq \s_1 +a_L\left(\e_j + \|w_j\|^{-1}\right)^{p-1},\qquad j\geq 1.
 \end{equation}
Arguing by contradiction, suppose that
$$
  \lim_{j\to \infty}\|w_j\|=\infty.
$$
Then, since $p-1<0$, letting $j\to \infty$ in \eqref{2.18} yields $\l_j\to\infty$ if $j\to \infty$, which contradicts $\l_j < b+1$. Consequently, there exists $M(b)$ satisfying \eqref{2.17}.
\par
For any given $d>0$, let  $B_d(0)$ denote the open ball of radius $d$ about $0$ in
$\mc{C}^1(\bar \O)$, and consider the open cylinder
$$
   D_{b} \equiv (0,b) \times B_M(0).
$$
This is a bounded open neighborhood of $(\s_1,0)$ in $\R \times \mc{C}^1(\bar \O)$. Therefore, there is a $(\mu, \varphi) \in \ms{S}_0\cap \p D_{b}$. According to \eqref{2.17}, $\mu = b$. Now we
need the following topological result from \cite{Why} --- see (9.3) on page 12.

\begin{lemma}
\label{le2.1}
If $X$ and $Y$ are disjoint closed subsets of a compact metric space, $K$, such that no component of $K$ intersects both $X$ and $Y$, then there is a separation of $K$ into $K _X$ and $K_Y$, i.e.
$$
   K = K_X \cup K_Y, \qquad  K_X \cap K_Y = \emptyset,
$$
and $K_X$, $K_Y$ are compact with $X \subset K_X$, $Y \subset K_Y$.
\end{lemma}

For our setting, $K = \ms{S} \cap \bar D_{b}$ under the induced topology from
$\R \times \mc{C}^1(\bar \O)$, and
$$
  X = \ms{S}_0,\qquad Y = \ms{S} \cap \left(\{b\} \times \mc{C}^1(\bar \O)\right).
$$
Lemma \ref{le2.1} provides the sets $K_X$ and $K_Y$. Since these later two sets are
compact and disjoint, the distance between them, $\d$, is positive. For $\omega < \delta / 3$,
let $N_{\omega}$ denote the intersection of a uniform open $\omega$ neighborhood of $K_X$ with $D_{b}$. Then, by Lemma \ref{le2.1},
\begin{equation}
\label{2.19}
\p (N_{\o}) \cap \ms{S} = \emptyset.
\end{equation}
But,  since $N_{\omega}$ is a bounded open set containing $(\s_1,0)$,
\begin{equation}
\label{2.20}
\p (N_{\omega}) \cap \ms{S}_0 \neq \emptyset,
\end{equation}
contrary to \eqref{2.19}. Thus, $\ms{S}_0$ cannot have a bounded projection on $\R$ and the
proof of Theorem \ref{th2.1} is complete. \hfill $\Box$
\par
\vspace{0.2cm}

\begin{remark}
\label{rm2.3}
It is natural to conjecture that as $\l \to \infty$, for any correspond $(\l,u)$ on $\ms{D}^+, \; u \to 0$.
Indeed that is shown in the next section for the simplest case of \eqref{1.1} when $N=1$. It remains
an open question in the more general settings.
\end{remark}

\section{The simplest one-dimensional prototype model}\label{section-3}

\noindent In this section \eqref{1.5} will be studied,
where $L>0$ and  $a>0$ are positive constants.
Our analysis will be based on a sharp phase plane analysis of
the underlying differential equation \eqref{1.3}.
Setting $u'=v$, \eqref{1.3} can be expressed as a first order system
\begin{equation}
\label{3.3}
  \left\{ \begin{array}{ll} u'= v, \\ v'=  a |u|^{p-1} u-\l u. \end{array}\right.
\end{equation}
Since the nonlinearity of \eqref{3.3} does not satisfy a Lipschitz
condition at $u=0$, one cannot expect the uniqueness of the
solution for the initial value problem associated with
\eqref{3.3}. However, as will become apparent later, except for
one particular value of the energy, there is uniqueness for the
associated Cauchy problems.
\par
Consider the function
\begin{equation}
\label{3.4}
   g(u)\equiv g_{\lambda,a}(u):=\lambda u - a |u|^{p-1}u \qquad (g(0)=0),
\end{equation}
and its primitive, or associated potential energy,
\begin{equation}
\label{3.5}
G(u)=G_{\lambda,a}(u):=\int_{0}^u g(s)\,ds = \lambda\frac{u^2}{2} - a \frac{|u|^{p+1}}{p+1}\,.
\end{equation}
Then, the system \eqref{3.3} is conservative with total energy, or
first integral given by
\begin{equation}
\label{3.6}
   E(u,v) = E_{\lambda,a}(u,v)= \frac{v^2}{2} + G(u),
\end{equation}
whose graph has been plotted in Figure  \ref{fig31}.

\begin{figure}[ht!]
    \centering
    \includegraphics[scale=0.5]{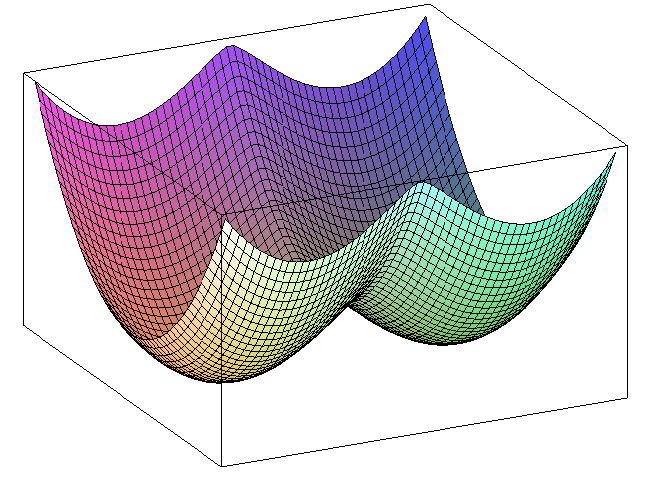}
    \caption{The surface $z=E(u,v)$.}
    \label{fig31}
\end{figure}

The system \eqref{3.3} has three equilibrium points, given  by the zeroes of $g(u)$. Namely,
$$0=(0,0) \quad \text{and }\;\; P_0^{\pm}=(u_0^{\pm},0), \quad
\text{for }\; u_0^{\pm}= u_0^{\pm}(\lambda):=
\pm\left(\frac{a}{\lambda}\right)^{\frac{1}{1-p}}.
$$
Figure \ref{fig32} shows the energy level lines of \eqref{3.6}. Every solution of \eqref{3.3}
lies on some energy level line.

\begin{figure}[ht!]
    \centering
    \includegraphics[scale=0.55]{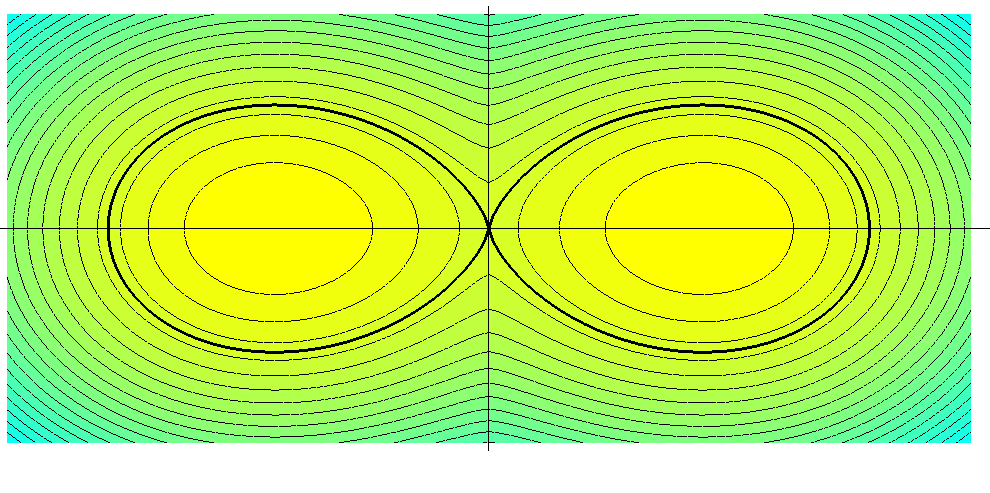}
    \caption{Level lines of the energy function. We have highlighted $[E=0]$. }
    \label{fig32}
\end{figure}

Next we consider the energy level passing through the origin
$$
  [E=0]:=\{(u,v)\in {\mathbb R}^2:\,E(u,v)=0\},
$$
which is expressed by the relation $v=\pm\sqrt{-2G(u)}$, where $G(u)$ is the potential energy.
By \eqref{3.5},
\begin{equation}
\label{3.7}
\sqrt{-2G(u)}=
\left(\frac{2a}{p+1}\right)^{1/2} |u|^{\frac{p+1}{2}}
\left(1- \frac{\lambda (p+1)}{2a} |u|^{1-p}\right)^{1/2}.
\end{equation}
The level set $[E=0]$ splits into 3 pieces:
$$
   [E=0]=\{(0,0)\} \cup \mathcal{O}^+ \cup \mathcal{O}^-,
$$
where
$$
  \mathcal{O}^+=\{y=\pm\sqrt{-2G(u)}: \quad 0< u\leq u_{H}\}
$$
with
\begin{equation}
\label{3.8}
u_H\equiv u_H(\lambda):= \left(\frac{2a}{\lambda(p+1)}\right)^{\frac{1}{1-p}}
= \frac{1}{\lambda^{\frac{1}{1-p}}}w_{H}, \qquad w_{H}:=\left(\frac{2a}{p+1}\right)^{\frac{1}{1-p}},
\end{equation}
and $\mathcal{O}^-$ is the symmetric of $\mathcal{O}^+$ with
respect to the $v$-axis. The sets $\mathcal{O}^{+}$ and
$\mathcal{O}^{-}$ are homoclinic loops around the equilibrium
points $P_0^{+}$ and $P_0^{-}$, respectively. As will be shown
later, these are degenerate homoclinic orbits in the sense that
they reach $(0,0)$ in a finite time because of the lack of a
Lipschitz condition at $u=0$.  This situation is reminiscent of
the one studied in \cite[Section 2]{CS} in the analysis of the
equation
$$
   -u''= s u^{\beta} - r u^{\alpha}
$$
for $r,s>0$ and $0<\alpha <\beta <1$. This equation arises in
modeling the Liebau phenomenon in blood circulation (see
\cite[Ch.8]{To}). The lack of Lipschitz continuity enables the
existence of solutions of \eqref{1.5} that vanish on some internal
subintervals of $(0,L)$ and are otherwise positive in $(0,L)$. It
shows that, even in the simplest prototypes of \eqref{1.1} one
cannot expect the solutions positive in $\Omega$ to satisfy $u>>0$.
In particular, the available maximum principle when $p\geq
1$ is lost for all $p\in (0,1)$. Naturally, such pathologies
straighten  the importance of Theorem \ref{th2.1}.
\par
The remaining solutions of \eqref{1.3} are uniquely determined by their initial conditions and
globally defined in time. This follows either by a direct inspection, or
using the results of \cite{R}, where the problem of the uniqueness for planar Hamiltonian
systems without a local Lipschitz condition for the vector field
was discussed. Therefore, the system  \eqref{3.3} defines a dynamical system
on the open set $\mathcal{A}= {\mathbb R}^2\setminus [E=0]$. Indeed, for every initial point $P\in \mathcal{A}$,
there is a unique solution,
$$
   z(\cdot;P)=(u(\cdot;P),v(\cdot;P))=(u(\cdot;P),u'(\cdot;P)),\qquad
z(0;P)=P
$$
of the system \eqref{3.3}, which is globally defined in time.
Actually, it is a periodic solution; possibly an equilibrium point
if $P=P_{0}^{\pm}$.
\par

Figure \ref{fig33} shows the graphs of three solutions $(x,u(x))$ of \eqref{1.3} with initial points
$(u_H,0),$ $(u_H-\varepsilon,0)$ and $(u_H+\varepsilon,0)$, for sufficiently small $\e>0$, where $u_H$ is the
positive abscissa of the crossing point of the homoclinic $\mc{O}^+$ with the $u$-axis, given in \eqref{3.8}.
Thus, the solution departing from $(u_H,0)$ corresponds to the homoclinic loop $\mathcal{O}^{+}$,
having $u_H$ as its maximum value and vanishing at a finite time $\pm T_{H}$
(the precise value of $T_{H}$ is given in \eqref{3.10}).
The solution departing from $(u_H-\varepsilon,0)$ is periodic and it oscillates around the positive equilibrium point
$P_0^{+}$, while the solution departing from $(u_H+\varepsilon,0)$ is periodic and it oscillates around the origin. This latter
solution, truncated on the open interval determined by its first left and right
zeros, provides us with a classical positive solution of \eqref{1.5}, with nonzero first derivative at the boundary points, for the appropriate value of $L>0$ (the distance between the two consecutive zeroes).

\begin{figure}[ht!]
    \centering
    \includegraphics[scale=0.45]{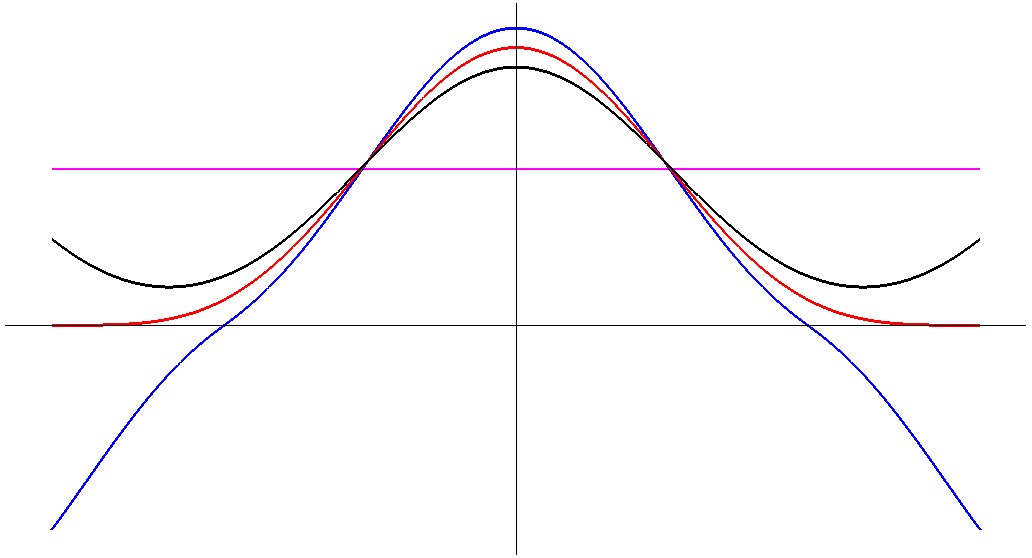}
    \caption{Three solutions of \eqref{1.3}. The simulation is made for
    the coefficients $\lambda=1,$ $a=2$ and $p=1/2.$ For the initial point we have
    taken $u_H=64/9$ (according to \eqref{3.8}) and $\varepsilon=1/2.$
    All these solutions have their first inflection point at level $u=u_0^+$
    (the abscissa of the nontrivial equilibrium point).
    }
    \label{fig33}
\end{figure}

Our analysis now proceeds by considering the three different situations illustrated in
Figure \ref{fig33}. Since we are interested in non-negative solutions,
we will restrict ourselves to the study of the solutions of \eqref{3.3}
in the half-plane $u\geq 0$.
For every $c>u_{H}$, the solution $u(\cdot;c):=u(\cdot;(c,0))$
of \eqref{1.3} with initial value $u(0)=c$ such that
$u'(0)=0$ is positive in a maximal open interval $I_c=(-T(c),T(c))$.  Moreover,
it vanishes with nonzero first derivative at $\pm T(c)$. Thus, these solutions
provide us with a positive solution of \eqref{1.5} for $L= 2T(c)$. In this paper, these strongly positive
solutions will be refereed to as classical positive solutions of \eqref{1.5}. They satisfy $u(x)>0$ for all $x\in (0,L)$, $u'(0)>0$
and $u'(L)<0$.
When some of these conditions fail but $u \not \equiv 0$ the solution
will be called degenerate. Thus,
if $u$ is a degenerate positive solution, then $u(x_0)=u'(x_0)=0$ for some $x_0\in [0,L]$. Naturally, a degenerate solution might vanish on some, or several, subintervals of $(0.L)$. Such positive solutions are said to be  highly degenerate.
\par
If, instead of  $c>u_{H}$, we pick $c\in (u_{0}^+,u_{H})$, then  $u(\cdot;c):=u(\cdot;(c,0))$
is a positive periodic solution oscillating around $P_0^+$, with half-period denoted by $\tau(c)$, such that
$\max u =c$ and $\min u = d(c)\in (0,u_0^+)$ with $G(d(c))=G(c)$. Thus, the equilibrium point $P_{0}^+$ is a (local) center and $\{(0,0)\}\cup \mathcal{O}^+$ is the boundary of the open region around
$P_{0}^+$ filled in by periodic orbits.
\par
Under these circumstances, since the orbits in $\mathcal{A}$
intersect the $v$-axis and
the line $u=u_{0}^+$ transversally, the Conley--Waz\.{e}wski theory
\cite{C} guarantees that the mappings
$(u_{H},+\infty)\ni c\mapsto T(c)$ and $(u_{0}^+,u_{H})\ni c\mapsto \tau(c)$ are continuous
(see also \cite[pp. 82-84]{HZ} for the details). However, the fact that
\begin{equation*}
\lim_{c\da u_H} T(c)=T(u_H)=T_H \quad \hbox{and}\quad \lim_{c\ua u_H} \tau(c)=T(u_H)=T_H
\end{equation*}
is far from obvious. It will be established later.
\par
As already commented above, the solution of \eqref{1.3} with
$u(0)=u_H$ and $u'(0)=0$, which parameterizes the degenerate homoclinic
loop $\mathcal{O}^+$, reaches the origin in a finite time. A
similar situation was discussed in \cite{PT} for $-u''=-a(x)u^p$, with $0<p<1$ and $a(x)$
positive and increasing. More precisely, in \cite{PT}
the existence of a  $\gamma>0$ for which
there are solutions of \eqref{1.3} with $u(0)=\gamma>0$ and $u'(0)=0$ such that
$u(\xi^*)=u'(\xi^*)=0$ for some $\xi^*=\xi(\gamma^*)$ was shown.
Some further developments in a similar vain
can be found in \cite{BPT}, in studying the non-negative solutions of
PDEs with a sign-indefinite weight, and in \cite{B}, where
solutions with the same property were found in
searching for  periodic solutions to the second-order equation
$-u''=a(x)g(u)$ with $g(u)u\geq 0$ and $a(x)$ changing sign, in
the ``sublinear case'' $g(u)/u\to +\infty$ for $u\to 0$.
The function $g(u)$ defined in \eqref{3.4}
does not satisfy these requirements.

\subsection{The degenerate homoclinic solution}\label{sub-7.1}
Consider the solution of \eqref{1.3} with $(u(0),u'(0))=(u_{H},0)$. As long as
$v(x)=u'(x)>0$, it follows from $u'(x) = \sqrt{-2G(u(x))}$ that the necessary time
to reach $u_2$ from $u=u_1$, with $0<u_1< u_2< u_{H}$, along $\mathcal{O}^+$ in the upper-half plane
is given through
\begin{equation*}
t(u_1,u_2)=\int_{u_1}^{u_2} \frac{du}{\sqrt{-2G(u)}}.
\end{equation*}
Letting $u_1\to 0^+$ and $u_2\to u_H^-$, and taking onto account that
$$
   u_{H}^{1-p}=\frac{2a}{\lambda(p+1)},
$$
it follows from \eqref{3.7} that the
vanishing time from the maximum point $u_{\max}=u_{H}$ to $u=0$ along $\mathcal{O}^+$
can be expressed through the \textit{convergent} improper integral
\begin{eqnarray*}
T_{H} & \equiv & t(0,u_H)=\int_{0}^{u_{H}} \frac{du}{\sqrt{-2G(u)}}=
\left(\frac{2a}{p+1}\right)^{-1/2}\int_{0}^{u_{H}} \frac{du}{u^{\frac{p+1}{2}}
\sqrt{1- \frac{\lambda (p+1)}{2a} u^{1-p}}}\\
&=&
\lambda^{-1/2}\left(\frac{2a}{\lambda(p+1)}\right)^{-1/2}
\int_{0}^{u_{H}} \frac{du}{u^{\frac{p+1}{2}}
\sqrt{1- u_{H}^{-(1-p)} u^{1-p}}}\\
&=&
\frac{1}{\sqrt{\lambda}} u_{H}^{-\frac{1-p}{2}}
\int_{0}^{u_{H}} \frac{du}{u^{\frac{p+1}{2}}
\sqrt{1- u_{H}^{-(1-p)} u^{1-p}}}.
\end{eqnarray*}
Thus, the change of variables $u=u_{H} s$, $0< u< u_{H}$, $0< s <1$, leads to
\begin{equation}
\label{3.9}
T_{H}= \frac{1}{\sqrt{\lambda}} u_{H}^{-\frac{1-p}{2}}
\int_{0}^{1} \frac{u_{H}\,ds}{u_H^{\frac{p+1}{2}} s^{\frac{p+1}{2}}
\sqrt{1- s^{1-p}}} = \frac{1}{\sqrt{\lambda}}
\int_{0}^{1} \frac{ds}{s^{\frac{p+1}{2}} \sqrt{1- s^{1-p}}}<\infty.
\end{equation}
By the symmetry of the problem, $T_H$ provides us with the extinction time
of the solution of \eqref{1.3} such that $u(0)=u_H$ and $u'(0)=0$.
Note that it is independent of the parameter $a$. The expression
\eqref{3.9} can be further simplified by the new change of variable
$$
   s^{1-p} = \sin^2\theta, \qquad 0<s<1,\quad 0< \theta <\frac{\pi}{2}.
$$
Indeed, since $(1-p)s^{-p} = 2 \sin \t \cos \t$, it is clear that
\begin{eqnarray*}
T_{H}&=&\frac{1}{\sqrt{\lambda}}
\int_{0}^{1} \frac{ds}{s^{\frac{p+1}{2}} \sqrt{1- s^{1-p}}}
=
\frac{1}{\sqrt{\lambda}(1-p)}
\int_{0}^{1} \frac{(1-p)s^{-p}\,ds}{s^{\frac{p+1}{2}}s^{-p}\sqrt{1- s^{1-p}}}\\
&=&
\frac{1}{\sqrt{\lambda}(1-p)}
\int_{0}^{1} \frac{(1-p)s^{-p}\,ds}{(s^{1-p})^{1/2}
\sqrt{1- s^{1-p}}}
=
\frac{1}{\sqrt{\lambda}(1-p)}
\int_{0}^{\pi/2} \frac{2 \sin\theta \cos\theta \,d\theta}{\sin\theta
\sqrt{1- \sin^2\theta}}.
\end{eqnarray*}
Therefore,
\begin{equation}
\label{3.10}
T_{H} = T_{H}(\lambda)= \frac{\pi}{\sqrt{\lambda}(1-p)}.
\end{equation}
Another way to prove \eqref{3.10} is to express $T_H$ as
\begin{equation}
\label{3.11}
   T_{H}\equiv T_H(\l)= \frac{1}{\sqrt{\lambda}}
\int_{0}^{1} \frac{ds}{s^{\frac{p+1}{2}} \sqrt{1- s^{1-p}}} =
\frac{1}{\sqrt{\lambda}}
\int_{0}^{1} \frac{ds}{s\sqrt{s^{p-1}- 1}}
\end{equation}
and observe that, up to an additive constant,
\begin{equation}
\label{3.12}
\int \frac{ds}{s\sqrt{s^{p-1}- 1}} =
-\frac{2\arctan\sqrt{s^{p-1}-1}}{1-p}.
\end{equation}
Finally, evaluating the definite integral of \eqref{3.11}, \eqref{3.10} follows readily.
\par
Figure \ref{fig34} illustrates the
independence of the extinction time $T_H$ with respect to the coefficient $a$. It plots
the graphs $(x,u(x))$ of the degenerate homoclinic solution on the interval $[-T_H,T_H]$
for a given value of $\lambda$ and several choices of the parameter $a$.

\begin{figure}[ht!]
    \centering
    \includegraphics[scale=0.45]{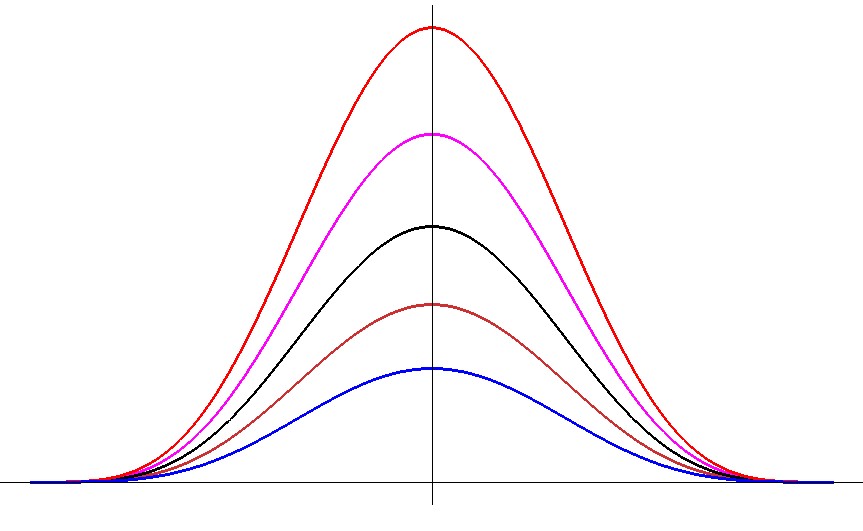}
    \caption{Five solutions of \eqref{1.3} having the same
    extinction time, $T_H$, for $(\lambda,p)$ constant and $a$ varying.
    The simulation has been made with $\lambda=1$ and $p=1/2$.  For the initial point we have
    taken $(u_H,0)$ (according to \eqref{3.8}). The solutions have been computed for the
     following series of values: $a=1$ ($u_H=16/9$), $a=5/4$ ($u_h=16/9$),
    $a=3/2$ ($u_H=4$), $a=7/4$ ($u_H=49/9$), $a=2$ ($u_H=64/9$).    }
    \label{fig34}
\end{figure}

Summarizing, the homoclinic loop $\mc{O}^+$ degenerates in the sense that the solution of \eqref{1.3}
with $u(0)=u_H$ and $u'(0)=0$, denoted by  $\hat u \equiv \hat u_\l$ in the sequel, runs through
the loop in a finite time
\begin{equation}
\label{eq-2TH}
2T_{H} \equiv 2T_{H}(\lambda)= \frac{2\pi}{\sqrt{\lambda}(1-p)}.
\end{equation}
This in strong contrast with the classical case of
$$
   -u''=-\lambda u + a|u|^{p-1}u
$$
when $p > 1$ and which has a similar portrait as in Figure \ref{fig32},
where an infinite amount of time is needed to traverse the loop. Note that
$\lim_{p\ua 1} T_H(\l)=\infty$ for all $\l>0$. Moreover, according to \eqref{3.8},
\begin{equation}
\label{eq-uH2}
\max \hat{u}_{\lambda}=u_H=\frac{1}{\lambda^{\frac{1}{1-p}}}w_{H},\qquad w_{H}:=\left(\frac{2a}{p+1}\right)^{\frac{1}{1-p}}.
\end{equation}
Since  \eqref{1.3} is autonomous, any shift in the $x$-variable also provides us with a solution
of the equation. Thus, for every $x_0\in\mathbb{R}$,
$\hat{u}_{\lambda}(\cdot-x_0)$ is a solution of \eqref{1.3} with
$u(x_0)=u_H$ and $u'(x_0)=0$, which is  unique on the interval
$$
   [x_0-T_H,x_0+T_H]={\rm supp\, } \hat{u}_{\lambda}(\cdot-x_0).
$$
In particular, since
\begin{equation}
\label{3.13}
   \lim_{\l\ua \infty}T_H(\l)=0,
\end{equation}
it follows that for every $x_0\in (0,L)$ and  $\l$
large enough that $0<x_0-T_H$ and $x_0+T_H<L$, the extended function
\begin{equation}
\label{3.14}
   \tilde u_{\l,x_0}(x) := \left\{ \begin{array}{ll} 0 & \quad \hbox{if}\;\; 0\leq x < x_0-T_H, \\
   \hat u_\l(x-x_0) & \quad \hbox{if}\;\; x_0-T_H\leq x \leq x_0+T_H, \\
   0 & \quad \hbox{if}\;\; x_0+T_H < x \leq L, \end{array}\right.
\end{equation}
is a positive solution of the nonlinear boundary value problem \eqref{1.5}, which is highly degenerate.
The precise values of $\l$ for which this construction can be done will be determined later. Thus, the maximum principle fails
for this type of sublinear nonlinearitiy.
Note that
due to \eqref{3.13}, for every $x_0\in (0,L)$,
\begin{equation}
\label{iii.15}
   \lim_{\l\ua \infty}\left( \lambda^{\frac{1}{1-p}} \tilde u_{\l,x_0}(x) \right) =
 \left\{ \begin{array}{ll} 0 & \quad \hbox{if}\;\; x\in [0,L]\setminus\{x_0\}, \\
  w_H = \left(\frac{2a}{p+1}\right)^{\frac{1}{1-p}} & \quad \hbox{if}\;\; x=x_0. \end{array}\right.
\end{equation}
Making the change of variable $v = \lambda^{\frac{1}{1-p}}u$, \eqref{1.5} becomes
\begin{equation}
\label{iii.16}
\left\{ \begin{array}{ll}-\frac{1}{\l}\D v=v-av^p&\quad \hbox{in}\;\;[0,L],\\[1ex]
v(0)=v(L)=0, & \end{array} \right.
\end{equation}
for $\D=\frac{d^2}{dx^2}.$ Then, \eqref{iii.15} provides us with
the limiting behaviour as $\l\ua \infty$ of the positive solutions
of the singular perturbation problem \eqref{iii.16}. Naturally,
for every integer $m\geq 1$, and $m$ given points $x_j\in (0,L)$,
$1\leq j\leq m$, such that
\begin{equation}
\label{iii.17}
  (x_i-T_H,x_i+T_H)\cap (x_j-T_H,x_j+T_H)=\emptyset \; \text{if }\; i\not=j,
\end{equation}
the superposition
\begin{equation}
\label{iii.18}
{u}_{\lambda,x_1,...,x_m}(x):=\sum_{j=1}^m \tilde{u}_{\lambda,x_j}(x),\qquad x\in [0,L],
\end{equation}
also provides us with a solution of \eqref{1.5} due to the assumption \eqref{iii.17}.
By \eqref{3.13}, the solution \eqref{iii.18} satisfies
\begin{equation*}
   \lim_{\l\ua \infty}\left( \lambda^{\frac{1}{1-p}} u_{\l,x_1,...,x_m}(x) \right) =
 \left\{ \begin{array}{ll} 0 & \quad \hbox{if}\;\; x\in [0,L]\setminus\{x_1,...,x_m\}, \\
  w_H = \left(\frac{2a}{p+1}\right)^{\frac{1}{1-p}} & \quad \hbox{if}\;\; x\in\{x_1,...,x_m\}. \end{array}\right.
\end{equation*}
Since $[0,L]$ is compact, this construction cannot be done if we choose a countable set of points $x_j\in (0,L)$, $j\geq 1$,
because the $x_j$'s would accumulate at some $x_0\in[0,L]$. Thus, \eqref{iii.17} cannot be satisfied. When
dealing with solutions in $\R$, a countable number of terms in the sum in \eqref{iii.18} is allowed provided that the points
 $x_j$, $j\geq 1$, are separated away from each other by a minimal distance.
Actually, as the solutions of \eqref{1.5} arise in pairs,
in the sense that $u$ solves \eqref{1.5} if, and only if,
$-u$ is also a solution, similar multiplicity results
can be given for sign-changing solutions of \eqref{1.5}. In this case, gluing together
positive and negative homoclinic loops,   solutions of \eqref{1.5} of the form
\begin{equation*}
{\mathscr{W}}(x)=\sum_{j=1}^m \delta_j\tilde{u}_{\lambda,x_j}(x), \qquad \d_j =\pm 1,
\end{equation*}
can be obtained provided \eqref{iii.17} holds.
Clearly, these degenerate situations are  due to the
lack of uniqueness for the associated Cauchy problem at the origin.
These results should be compared with the gluing of variational solutions found in \cite{MR}.
\par
Figure \ref{fig-7.5} shows a
family of nested homoclinics for $\lambda$ ranging from a smaller value (the external orbit)
to a larger one (the smaller orbit).
This displays the dependence of $\max \hat{u}_{\lambda}$ on the parameter $\lambda$.

\begin{figure}[ht!]
    \centering
    \includegraphics[scale=0.45]{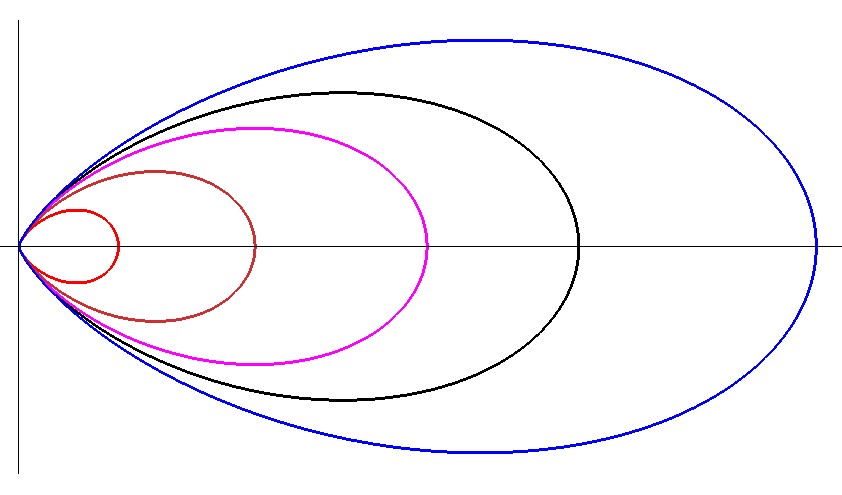}
    \caption{Different atypical homoclinics of system \eqref{3.3} for a fixed pair
    $(a,p)$ and different values of $\lambda$.
    The simulation is made for
    the coefficients $a=2$ and $p=2/3.$ The values chosen for $\lambda$ vary from
    $\lambda=1$ the external orbit) to $\lambda=2$ (the smallest orbit).
    }
    \label{fig-7.5}
\end{figure}

We conclude our analysis of these degenerate homoclinics
by providing an explicit expression for $\hat{u}_{\lambda}$ as follows.
The same argument as in
the beginning of Section 3.1 shows that the
distance, $|x|$, from $u(x)=u_1$ and $u(0)=\max u = u_H$ is given by
$$|x|=\int_{u(x)}^{u_H} \frac{d\xi}{\sqrt{-2G(\xi)}}.$$
Thus, repeating the same change of variables leading to
\eqref{3.9}, $\xi = u_H s$,  yields
$$
|x|= \frac{1}{\sqrt{\lambda}}
\int_{\frac{u(x)}{u_H}}^{1} \frac{ds}{s^{\frac{p+1}{2}}
\sqrt{1- s^{1-p}}}.
$$
Thus  it follows from  \eqref{3.11} and \eqref{3.12} that
$$
   |x| =\frac{2}{(1-p)\sqrt{\l}}\arctan \sqrt{\left(\tfrac{u(x)}{u_H}\right)^{p-1}-1}
$$
and hence,
$$
  \left(\frac{u(x)}{u_H}\right)^{p-1} =
   1 + \tan^2 \frac{(1-p)\sqrt{\lambda}|x|}{2}.
$$
Finally, because of \eqref{3.8}, we obtain the explicit expression
$$
\hat{u}_{\lambda}(x)= \left(\frac{2a}{\lambda(p+1)
\left(1 + \tan^2 \frac{(1-p)\sqrt{\lambda}|x|}{2}\right)}\right)^{\frac{1}{1-p}}
$$
for every $-T_H<x<T_H$, $T_H=T_{H}(\lambda)$. By construction, $\hat{u}_{\lambda}(x)>0$ for all $x\in (-T_H,T_H)$ and
\begin{equation}
\label{iii.19}
  \lim_{x\to \pm T_H} \hat{u}_{\lambda}(x)= 0, \qquad \lim_{x\to \pm T_H} \hat{u}'_{\lambda}(x)= 0.
\end{equation}

\subsection{The classical positive solutions}\label{sub-7.2}
This section characterizes the existence of  \emph{strongly
positive solutions} of \eqref{1.5}, also referred to  as
\emph{classical positive solutions} in this paper. Namely they are
solutions satisfying
\begin{equation}
\label{3.18}
  u(x)>0 \quad \hbox{for all}\;\; x\in (0,L),\quad u'(0)>0,\quad u'(L)<0.
\end{equation}
By the analysis already done in Section 3.1, they are given
by the solutions of \eqref{1.3} with $u(0)=c$ and $u'(0)=0$
for some $c>u_H$, and they are defined in some
interval $(-T(c),T(c))$ with $u(\pm T(c))=0$, $u'(-T(c))>0$, $u'(T(c))<0$, $T(c)=L/2$, and $u(x)>0$ for all $x\in (-T(c),T(c))$.
In the interval $(-T(c),0)$, where $u(x)>0$ and $u'(x)>0$,
the solution lies on the energy level
$E(u,v)=E(c,0)= G(c)$. From
$$
   u'(x) = \sqrt{2(G(c)-G(u(x)))},
$$
as long as $u'(x)=v(x)>0$, it is easily seen that the necessary ``time''
to reach $u=u_1\in (0,c)$ from $u=0$
along the orbit in the upper-half plane is given by
$$
   t(0,u_1)=\int_{0}^{u_1} \frac{du}{\sqrt{2(G(c)-G(u))}}.
$$
Thus, letting $u_1\ua c$ and recalling the definition of $G(u)$,  allows us to conclude that the
vanishing time $T(c)$ from the maximum point $\max u =c$ to $u=0$ along the level line passing through $(c,0)$
can be expressed via the convergent improper integral
\begin{equation}
\label{iii.20}
T(c)=\int_{0}^{c} \frac{du}{\sqrt{2(G(c)-G(u))}}=
\int_{0}^{c} \frac{du}
{\sqrt{\lambda(c^2-u^2) - \frac{2a}{p+1}\bigl(c^{p+1}-u^{p+1}\bigr)}}.
\end{equation}
(See  \cite{LMZ} for similar computations). Performing the change of variables $u=c s$, $0< u< c$, $0< s <1$, it follows from \eqref{iii.20} that
\begin{equation*}
T(c)=
\int_{0}^{1} \frac{c\,ds}
{\sqrt{\lambda c^2(1-s^2) - \frac{2ac^{p+1}}{p+1}\bigl(1-s^{p+1}\bigr)}}
= \frac{1}{\sqrt{\lambda}}\int_{0}^{1} \frac{ds}
{\sqrt{(1-s^2) - \frac{2a}{\lambda(p+1)}\frac{1}{c^{1-p}}\bigl(1-s^{p+1}\bigr)}}.
\end{equation*}
This equation together with \eqref{3.8} yields
\begin{equation}
\label{c.21}
T(c)= \frac{1}{\sqrt{\lambda}}\int_{0}^{1} \frac{ds}
{\sqrt{(1-s^2) - \left(\frac{u_H}{c}\right)^{1-p}\bigl(1-s^{p+1}\bigr)}}.
\end{equation}
Observe that, since $0<p<1$, we have that $1-s^2 > 1-s^{p+1}$ for all $0<s<1$.
Therefore
the mapping $(u_{H},+\infty)\ni c\mapsto T(c)$ is strictly decreasing.
For this particular case,
the continuity
of $T$ with respect to $c$ also follows directly from \eqref{c.21} without
having to invoke to any general result
on dynamical systems. Note that
$$
   \lim_{c\ua +\infty} T(c) = \frac{1}{\sqrt{\lambda}}\int_{0}^{1} \frac{ds}
{\sqrt{1-s^2}}= \frac{\pi}{2\sqrt{\lambda}}
$$
while, due to \eqref{3.9} and \eqref{3.10},
\begin{align*}
   \lim_{c\da u_H^+} T(c) &  = \frac{1}{\sqrt{\lambda}}\int_{0}^{1} \frac{ds}
{\sqrt{(1-s^2) - \bigl(1-s^{p+1}\bigr)}}\\ & =
\frac{1}{\sqrt{\lambda}} \int_{0}^{1}
\frac{ds}{s^{\frac{p+1}{2}}\sqrt{1- s^{1-p}}} = T_{H} =
\frac{\pi}{\sqrt{\lambda}(1-p)}.
\end{align*}
Thus, the map $T(c)$ extends continuously to $[u_{H},+\infty)$  by setting $T(u_{H})= T_{H}$.  Figure \ref{Fig6}
shows the graph of the mapping $c\mapsto T(c)$. In the light of these facts,
it is clear that the boundary value problem \eqref{1.5} has a
strongly positive solution if, and only if,
\begin{equation}
\label{c.22}
   \lim_{c\ua +\infty} T(c) = \frac{\pi}{2\sqrt{\l}} < \frac{L}{2}<\frac{\pi}{(1-p)\sqrt{\l}}=T_H.
\end{equation}
Moreover, this solution is unique, because there is a unique value of $c>u_H$, say $c(\l)$, such that $T(c(\l))=L/2$.
When this occurs, the unique positive solution of \eqref{1.5} is given by the $L/2$-shift of the unique positive solution of \eqref{1.3} such that $u(L/2)=c(\l)$ and $u'(L/2)=0$.

\begin{figure}[ht!]
    \centering
    \includegraphics[scale=0.4]{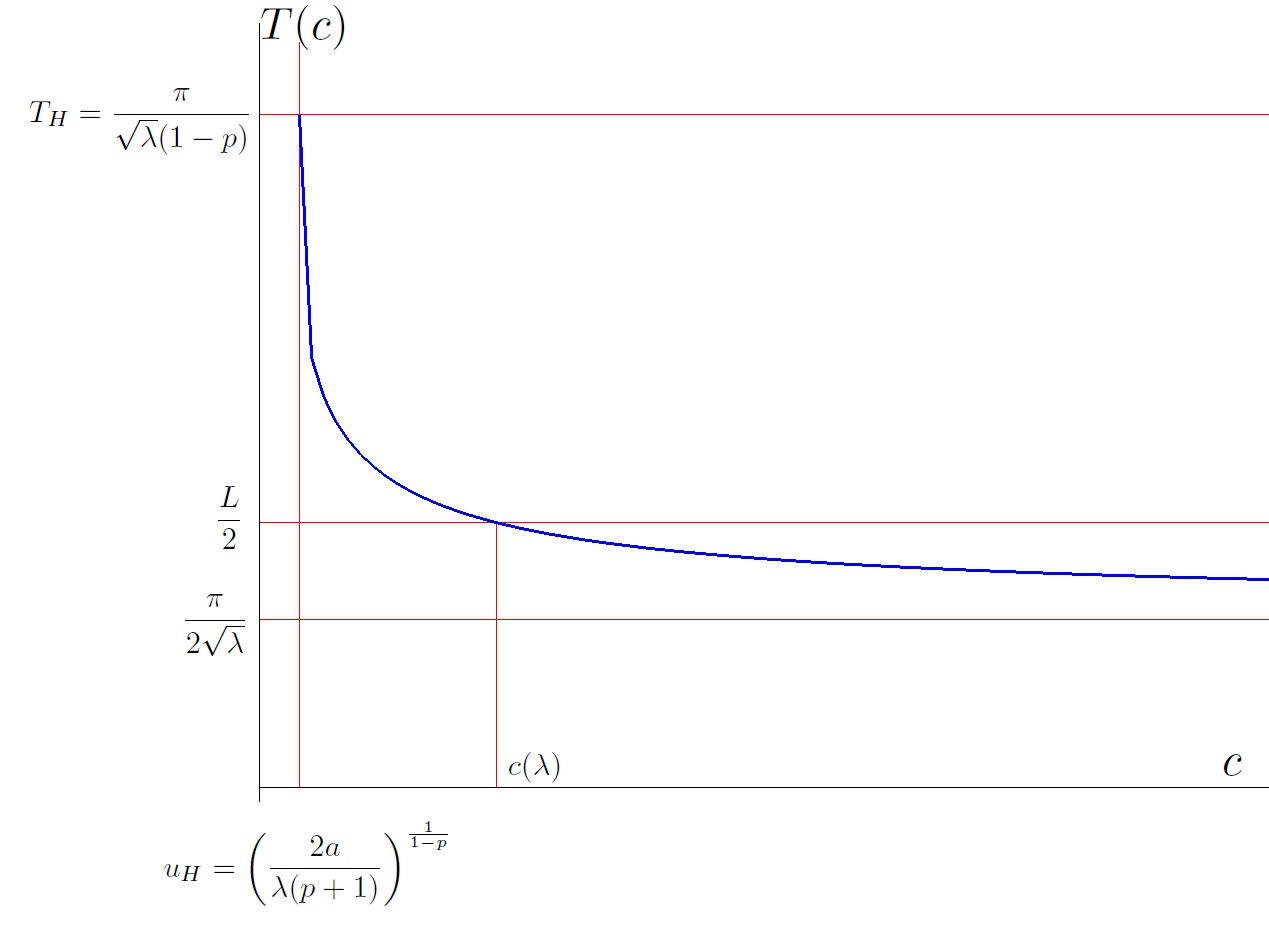}
    \caption{The graph of the time map $T(c)$.
    It is a decreasing function defined in $[u_H,\infty)$.
    By construction, the problem \eqref{1.5} has a classical
    positive solution if, and only if, \eqref{c.22} holds.
    The graph comes from a numerical simulation for the case $a=1,$ $\lambda=1$ and $p=1/2.$}
    \label{Fig6}
\end{figure}

Since \eqref{c.22} can be written equivalently as
\begin{equation}
\label{iii.23}
  \s_1=\left( \frac{\pi}{L}\right)^2 <\l < \left(\frac{2}{1-p}\right)^2\s_1\equiv \Sigma_1(p),
\end{equation}
the next result holds.

\begin{theorem}
\label{th3.1}
The problem \eqref{1.5}  has a classical positive solution if, and only if, $\l\in (\s_1,\Sigma_1(p))$, i.e.,
if \eqref{iii.23} holds. When \eqref{iii.23} holds, the solution is unique. Moreover, denoting it by ${u}_{\lambda}$, it satisfies
\begin{equation*}
   \max {u}_{\lambda} = {u}_{\lambda}(L/2) = c(\lambda),
\end{equation*}
and  the mapping $(\sigma_1,\Sigma_1(p))\ni\lambda \mapsto c(\lambda)$ is continuous and
strictly decreasing with
\begin{equation}
\label{iii.24}
   \lim_{\lambda\da \sigma_1} c(\lambda) =+\infty, \qquad
\lim_{\lambda\ua \Sigma_1(p)} c(\lambda)
=u_H(\Sigma_1(p))=\left(\frac{2a}{\Sigma_1(p)(p+1)}\right)^{\frac{1}{1-p}}.
\end{equation}
\end{theorem}
\begin{proof}
We already know that $\l\in (\s_1,\Sigma_1(p))$ is necessary and sufficient for the existence
of a $c(\l)>u_H$ such that $T(c(\l))=L/2$. Moreover, $c(\l)$ is unique by the monotonicity of $T(c)$. Thus, the classical positive solution, denoted by $u_\l$,
is unique and it satisfies $u_\l(L/2)=c(\l)$.
It also follows from \eqref{iii.20} that $T(c)$ is a continuous function of $\l$.
Thus, the definition of $c(\l)$ guarantees the continuous dependence of
$c(\l)$ with respect to $\l$.
\par
Next observe that due to \eqref{iii.20}, the identity  $T(c)=L/2$ holds for some $c>u_H$  if, and only if,
\begin{equation*}
\lambda=\Psi(u_H/c),
\end{equation*}
where $\Psi$ stands for the map
\begin{equation}
\label{iii.25}
\Psi(\xi):= \left(\frac{2}{L}\right)^2 \left(\int_{0}^{1}
\frac{ds}{\sqrt{(1-s^2) - \xi^{1-p}\bigl(1-s^{p+1}\bigr)}}\right)^2,
\qquad 0< \xi < 1.
\end{equation}
From \eqref{iii.25}, since $p<1$,  it is easily seen that $\Psi: (0,1)\to \mathbb{R}$ is strictly
increasing, with
\begin{equation}
\label{iii.26}
\lim_{\xi\da 0}\Psi(\xi)= \left(\frac{2}{L}\right)^2 \left(\frac{\pi}{2}\right)^2
= \left(\frac{\pi}{L}\right)^2 = \sigma_1.
\end{equation}
Moreover using \eqref{3.9} and \eqref{3.10}, we find that
\begin{equation}
\label{iii.27}
\lim_{\xi\ua 1}\Psi(\xi)= \left(\frac{2}{L}\right)^2
\left(\int_{0}^{1}
\frac{ds}
{\sqrt{s^{p+1}-s^2}}\right)^2 =
\left(\frac{2}{L}\right)^2\left(\frac{\pi}{1-p}\right)^2
=\left(\frac{2}{1-p}\right)^2\sigma_1=\Sigma_1(p).
\end{equation}
In particular, $\Psi: (0,1)\to \Psi(0,1)=(\sigma_1,\Sigma_1(p))$.
Finally, \eqref{iii.24} follows readily from \eqref{iii.26} and \eqref{iii.27}.
\end{proof}

Note that, since $\lambda\mapsto c(\l)$ is continuous,
the
map $\lambda\mapsto {u}_{\lambda}\in C^1_{0}[0,L]$ is also continuous. Moreover, by the uniqueness of $u_\l$, it follows from \eqref{iii.24} that in the interval $(\s_1,\Sigma_1(p))$, the component $\ms{D}^+$  whose existence was established by Theorem \ref{th2.1} consists of the continuous curve of classical positive solutions defined by
$$
  \mathscr{C}^+:= \{(\lambda,{u}_{\lambda})\; :\; \l\in (\s_1,\Sigma_1(p))\}.
$$
By the construction carried out in Sections 3.1 and 3.2, these classical positive solutions approximate the degenerate solution $\hat u$ as $\l\ua \Sigma_1(p)$. Actually, $\hat u$ is the unique solution of \eqref{1.3} satisfying
$u(0)=u_H$ and $u'(0)=0$. According to \eqref{iii.19}, we already know that $T(u_H)=L/2$, $\hat u(x)>0$ for all $x\in (-L,L)$, $\hat u(-L)=\hat u(L)=0$, and $\hat u'(-L)=\hat u'(L)=0$.
Note that  $\lambda=K_1(p)$ if, and only if, $T_H(\lambda)=L/2$.
Naturally, since \eqref{1.3} is autonomous, the shift
\begin{equation}
\label{eq-hatu1}
    {u}_{\Sigma_1(p)}(x):=\hat{u}\bigl(x-\tfrac{L}{2}\bigr), \qquad x\in [0,L],
\end{equation}
provides us with a degenerate positive solution of \eqref{1.5}.
Therefore, in the interval $(\s_1,\Sigma_1(p)]$
the component $\ms{D}^+$ consists of $\ms{C}^+$
plus the degenerate positive solution $(\Sigma_1(p),u_{\Sigma_1(p)})$.

\subsection{Highly degenerate positive solutions}\label{sub-7.3}
Next, for every $\l > \Sigma_1(p)$, we consider the highly
degenerate positive solutions $(\l,u_\l)$, with
$$
  u_\l(x):= \tilde u_{\l,L/2},
$$
where $\tilde u_{\l,L/2}$ stands for the  solution defined in \eqref{3.14} with
$x_0=L/2$. Note that $\l > \Sigma_1(p)$ if, and only if, $T_H<L/2$. By construction, since $\ms{D}^+$ is connected, $(\l,u_\l)\in\ms{D}^+$ for all $\l>\Sigma_1(p)$, by the continuity of the map $\l\mapsto u_\l$, $\l>\s_1$. Therefore,
$$
  \mathscr{C}^+:= \{(\lambda,{u}_{\lambda})\; :\; \l >\s_1\}\subset \ms{D}^+.
$$
The solutions on the curve $\ms{C}^+$ satisfy the following properties:

\begin{itemize}
\item  $\mathscr{C}^+$ is the graph of a continuous curve
$\lambda\mapsto {u}_{\lambda}$ which is defined for all $\lambda >\sigma_1$;

\item the map $\lambda\mapsto \|{u}_{\lambda}\|_{\infty}$ is strictly
decreasing, with
$$
  \lim_{\lambda\da \sigma_1}\|u_{\lambda}\|_{\infty}=+\infty\quad \hbox{and}\quad
  \lim_{\l\ua\infty}\|u_\l\|_\infty=0.
$$
Moreover, \eqref{iii.15} holds.

\item ${u}_{\lambda}$ is a classical positive  solution of \eqref{1.5} if, and only if,
$\l\in (\sigma_1,\Sigma_1(p))$.  Moreover, $\ms{C}^+=\ms{D}^+$ in $(\s_1,\Sigma_1(p))$.

\item At $\lambda=\Sigma_1(p)$,  ${u}_{\lambda}$ satisfies $u_\l(x)>0$ for all $x\in (0,L)$, $u_\l(0)=u_\l(L)=0$, and
$u_\l'(0)=u_\l'(L)=0$. Thus, this solution degenerates.

\item For every $\lambda>\Sigma_1(p),$ the solution ${u}_{\lambda}$ is symmetric about $x=L/2$ and
$\max {u}_{\lambda} = {u}_{\lambda}(L/2)$. Moreover,
$$
  {u}_{\lambda}(x)>0\quad \hbox{for all}\quad x\in (\tfrac{L}{2}-T_H(\lambda),\tfrac{L}{2}+T_H(\lambda)),
$$
while
$$
 {u}_{\lambda}\equiv 0\quad \hbox{on} \quad [0,L] \setminus (\tfrac{L}{2}-T_H(\lambda),\tfrac{L}{2}+T_H(\lambda)).
$$
Figure \ref{fig37} shows some of those  solutions.
\end{itemize}

\begin{figure}[ht!]
    \centering
    \includegraphics[scale=0.45]{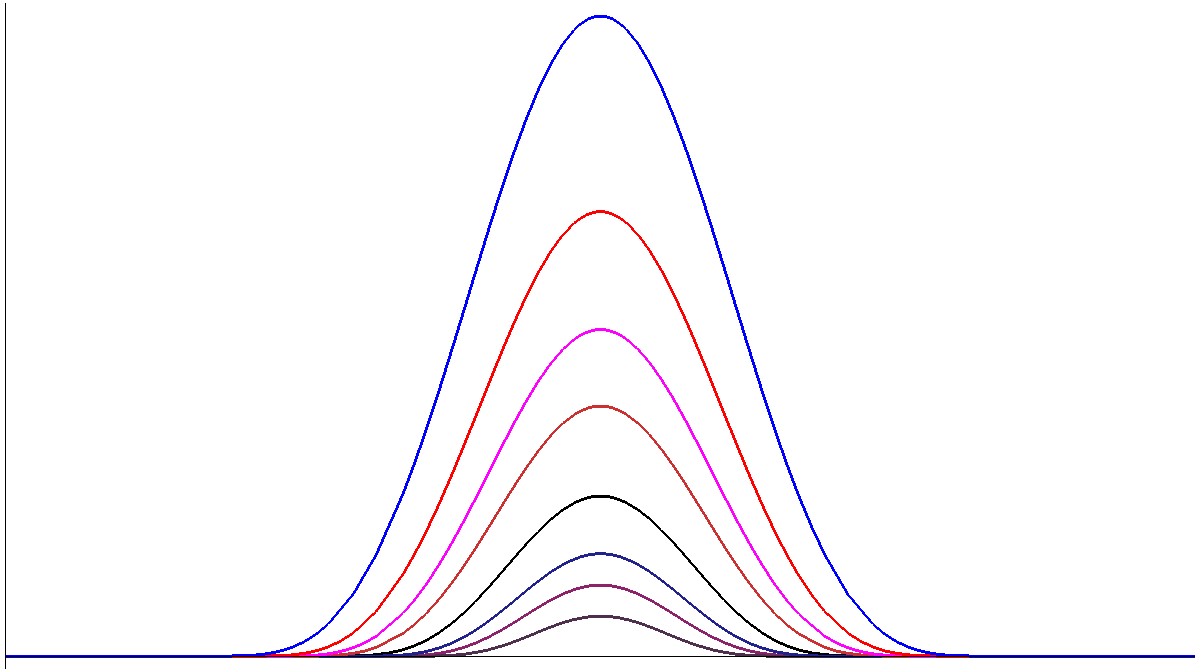}
    \caption{ The  plots of some highly degenerate solutions of \eqref{1.5}  for a series of values of $\l\in[1,4]$.
    The solutions decreased as $\l$ increased. The numerical experiment has been done with  $a=2$ and $p=1/2$.     We have taken $L = 6\pi > 4\pi=2T_H(1)$ to guarantee that  $u_\l$ be highly degenerate by choosing $\l \geq 1 > K_1(p)$.  }
    \label{fig37}
\end{figure}

Although $\ms{D}^+$ consists of a continuous curve
on $(\s_1,\Sigma_1(p))$,
it will be shown that $\ms{D}^+$ has a more complex structure than
$\ms{C}^+$ for $\l>\Sigma_1(p)$. Namely, for every $\l>\Sigma_1(p)$,
it contains
an additional continuous curve which can be constructed as follows.
For every $\l>\Sigma_1(p)$ and
$$
   \vartheta \in \left[  -\tfrac{L}{2} + T_H(\lambda),\tfrac{L}{2} - T_H(\lambda)\right],
$$
there is a new solution
$$
   {w}_{\lambda,\vartheta}(x):= {u}_{\lambda}\bigl(x-\vartheta\bigr), \quad x\in [0,L].
$$
Thus we have a new one-parameter family of highly degenerate solutions of \eqref{1.5}.
Since $\ms{D}^+$ is connected, and $w_{\l,0}=u_\l\in \ms{D}^+$,
it is clear that
$$
  w_{\l,\vartheta}\in \ms{D}^+\quad \hbox{for all}\quad  \l >\Sigma_1(p)\quad \hbox{and}\quad
  \vartheta \in \left[  -\tfrac{L}{2} + T_H(\lambda),\tfrac{L}{2} - T_H(\lambda)\right].
$$
Figure \ref{fig38} shows a
few solutions $w_{\lambda,\vartheta}$ for a fixed $\lambda >\Sigma_1(p)$ and
a series of values of $\vartheta$ varying from  $\vartheta^-\equiv \vartheta^-_{\lambda}:=
-\tfrac{L}{2} + T_H(\lambda)$ to $\vartheta^+\equiv \vartheta^+_{\lambda}:=\tfrac{L}{2} - T_H(\lambda)$.
For $\vartheta=0$, we get the \lq\lq central\rq \rq \, symmetric solution
$w_{\l,0}=u_\l\in\ms{C}^+\subset \ms{D}^+$.
\par
By construction, for $\vartheta=\vartheta^-$, we have that
$$
    {w}_{\lambda,\vartheta^-}(x) = {u}_{\lambda}(x-T_H),\qquad x \in [0,L].
$$
This function satisfies
$$
  {w}_{\lambda,\vartheta^-}(0)={w}_{\lambda,\vartheta^-}'(0)=0, \quad {w}_{\lambda,\vartheta^-}(x)>0 \;\;
  \hbox{for all}\;\; x\in (0,2T_H),\quad \hbox{and }\;\; {w}_{\lambda,\vartheta^-}=0 \;\; \hbox{on}\;\;
  [2T_H,L].
$$
Similarly, for $\vartheta=\vartheta^+$, we have that
$$
    {w}_{\lambda,\vartheta^+}(x) = {w}_{\lambda}(x-L+T_H),\qquad x\in [0,L],
$$
and this function satisfies
$$
  {w}_{\lambda,\vartheta^+}=0 \;\; \hbox{on}\;\; [0,L-2T_H],\quad {w}_{\lambda,\vartheta^+}(x)>0 \;\;
  \hbox{for all}\;\; x\in (L-2T_H,L),\quad \hbox{and }\;\;
  {w}_{\lambda,\vartheta^+}(L)={w}_{\lambda,\vartheta^+}'(L)=0.
$$

\begin{figure}[ht!]
    \centering
    \includegraphics[scale=0.4]{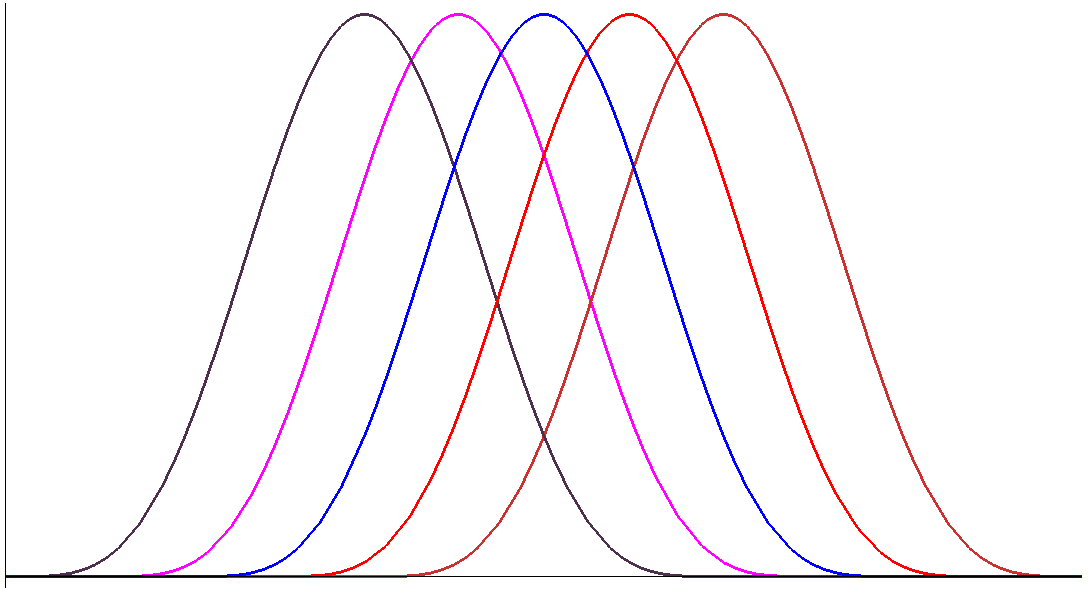}
    \caption{ The positive solutions ${w}_{\lambda,\vartheta}$ for a series of values of $\vartheta$ varying from $\vartheta^-_{\lambda}$ to $\vartheta^+_{\lambda}$. The simulation has been performed with $a=2$, $\lambda=1$, and $p=1/2$.  In this case, $2T_{H}(\lambda)= 4\pi$. We have taken $L = 6\pi > 2T_{H}(\lambda)$ so that
    $\lambda=1 > K_1(p)$.  The figure shows some shifts of the central symmetric solution, $u_\l$ for some values of the secondary parameter $\vartheta$.   }
    \label{fig38}
\end{figure}

The topological structure of the curve $\mathscr{C}^+$ together
with the shifted solutions $w_{\l,\vartheta}$ for all
$\l>\Sigma_1(p)$ has been sketched in Figure \ref{fig39}. All
these solutions are part of the component $\ms{D}^+$ whose
existence was established by Theorem \ref{th2.1}.  Although, for
every $\l\in (\s_1,\Sigma_1(p)]$, $\ms{D}^+$ consists of
$(\l,u_\l)$, we see that for every $\l>\Sigma_1(p)$, the
component $\ms{D}^+$ contains a one-dimensional simplex made by a
segment of shifted solutions  from $u_\l$ in the $x$-component.
Thus, for every $\l>\Sigma_1(p)$, \eqref{1.5} has a continuum of
positive highly degenerate solutions.

\begin{figure}[ht!]
    \centering
    \includegraphics[scale=0.45]{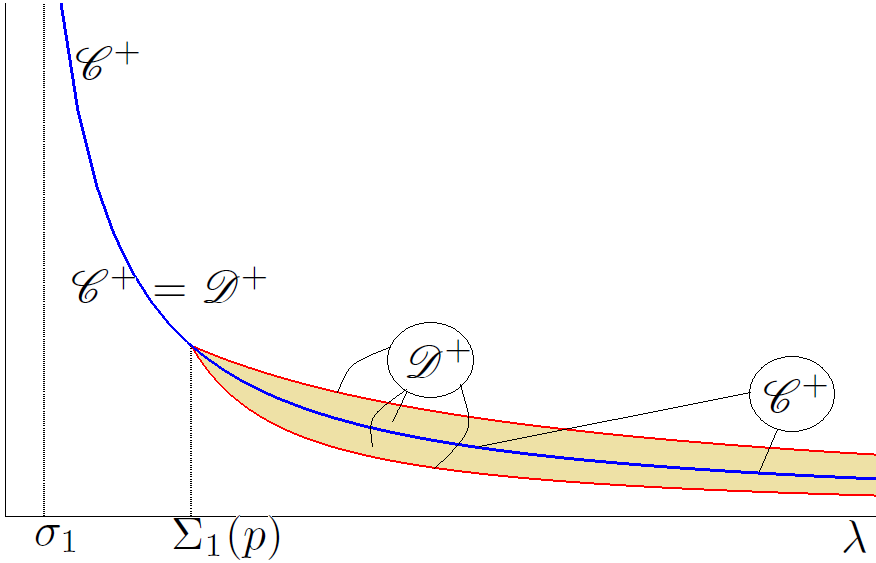}
    \caption{ The topological structure of the component $\ms{D}^+$. It is generated by the central fiber
    $\ms{C}^+=\{(\l,u_\l)\;:\; \l>\s_1\}$ through the shifts $u_{\l,\vartheta}$ for $\l>\Sigma_1(p)$. }
    \label{fig39}
\end{figure}

\subsection{Multibump solutions}\label{sub-7.4}
All the solutions constructed in the previous two sections have a
single bump in  $(0,L)$. In order to complete the classification
of the positive solutions of \eqref{1.5},  both classical and
degenerate, we now consider the possibility of  solutions with
multiple bumps in the interval $[0,L]$. In terms of the dynamics
in the phase plane, these solutions can be described as
trajectories making multiple transitions of the homoclinic loop
$\mathcal{O}^{+}.$ These transitions are separated from each other
by periods of ``rest'' where they remain at the origin. One, or
several of these rest intervals might shrink to a single point.
Since the time required to complete a homoclinic loop is
$2T_H(\lambda)$, with $T_H(\lambda)$ given by \eqref{3.10}, $j$
non-overlapping copies of the function ${u}_{\lambda}$ can exist
in (0,L) if, and only if,
\begin{equation}
\label{eq-Kj}
  2T_H(\l)\leq \frac{L}{j}\quad \Leftrightarrow \quad
  \l\geq  \Sigma_j(p):= \left(\frac{2}{1-p}\right)^2 \sigma_j,
\end{equation}
where $\sigma_j:= (j\pi/L)^2$ is the $j$-th eigenvalue of $-D^2$ in $[0,L]$ under homogeneous Dirichlet
boundary conditions. Figure \ref{fig310} illustrates the case $\lambda=K_2(p)$, plotting a solution of
\eqref{1.5} with  two positive humps in $(0,L)$. This solution vanishes with its first derivative at $x\in\{0,L/2,L\}$.

\begin{figure}[ht!]
    \centering
    \includegraphics[scale=0.45]{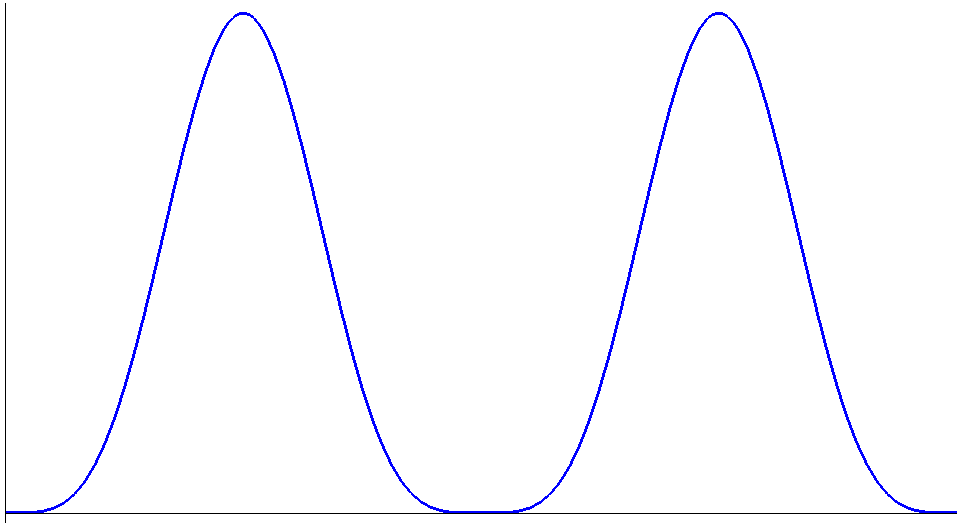}
    \caption{    A solution with two bumps computed for the choices $a=2$, $\lambda=1$ and $p=1/2$.
    We have taken $L=4 T_{H}(\lambda)$ so that $\lambda=K_2(p)$.  }
    \label{fig310}
\end{figure}

Next the structure of the component of (degenerate) positive
solutions with $j\geq 2$  bumps of the problem \eqref{1.5} will be
described. First, fix $j\geq 2$, suppose $\l\geq \Sigma_j(\l)$,
and consider $j$  points
$$
  0<x_1<x_2<\cdots <x_j<L
$$
such that
\begin{equation}
\label{deltas}
\begin{split}
   \delta_1 & :=x_1- T_H(\lambda)\geq 0,\\
     \delta_i & := x_i-x_{i-1}- 2T_H(\lambda)\geq 0 \;\; (2\leq i\leq j),\\
  \delta_{j+1} & :=L-x_{j}-T_H(\lambda)\geq 0,\\
   \sum_{i=1}^{j+1}\delta_i & = L-2 j T_H(\lambda).
\end{split}
\end{equation}
Then by construction, the function
\begin{equation}
\label{iii.28}
 {u}_{\lambda,x_1,\dots,x_j}(x)=
\sum_{i=1}^{j}\tilde{u}_{\lambda,x_i}(x),\qquad x\in [0,L],
\end{equation}
provides us with a degenerate positive solution of \eqref{1.5} having a bump at each of the points
$x_j$. In \eqref{iii.28}, $\tilde{u}_{\l,x_i}$ is the degenerate positive solution defined
in \eqref{3.14}, with $x_0=x_i$.
\par
The following  result shows that, at any of these bumps, the
positive solution reaches
$$
   u_H(\l) \equiv \left( \frac{2a}{(p+1)\l}\right)^\frac{1}{1-p}.
$$
The proof is omitted, since it is a direct consequence of the analysis
previously performed on the homoclinic loops and it also comes from formula
\eqref{3.8}.

\begin{lemma}
\label{le3.1}
Suppose $0\leq \a<\b\leq L$ and $u$ is a positive solution of
\begin{equation*}
  \left\lbrace\begin{array}{l}  -u''=\l u -a |u|^{p-1}u  \quad \hbox{in }[\a,\b],\\
  u(\a)=u(\b)=0,
  \end{array}\right.
\end{equation*}
such that $u(x)>0$ for all $x\in (\a,\b)$. Then,
$$
   \|u\|_{\infty,[\a,\b]} = u(\tfrac{\a+\b}{2}) \geq
    \left(\tfrac{2a}{(p+1) \l}\right)^\frac{1}{1-p}
$$
\end{lemma}

As a byproduct, for every $i, j \geq 1$ with $i\neq j$, the components of solutions with $i$ bumps cannot touch
any of the components with $j$ bumps. Actually, adapting the argument of Section 3.3, for every $j\geq 2$
and $\l>\Sigma_j(p)$ (fixed), one can construct from $u_{\l,x_1,...,x_j}$  a $j$--dimensional simplex consisting of
homotopic solutions to  $u_{\l,x_1,...,x_j}$  by slightly changing some, or several,
of the points $x_i$, $1\leq i\leq j$. The underlying simplex of solutions for a given $\l>\Sigma_j(p)$
expands as $\l$ increases, as sketched in Figure \ref{fig310b}, where, for evey $j\geq 2$, we have denoted by
$\ms{D}_j^+$ the component of highly degenerate positive solutions of \eqref{1.5} with $j$ bumps in $(0,L)$.
By the construction, it is easily seen that, necessarily, $\ms{D}^+_j$ is the unique component of positive
solutions with $j$ bumps, since any pair of solutions of this type are homotopic.

\begin{figure}[ht!]
    \centering
    \includegraphics[scale=0.45]{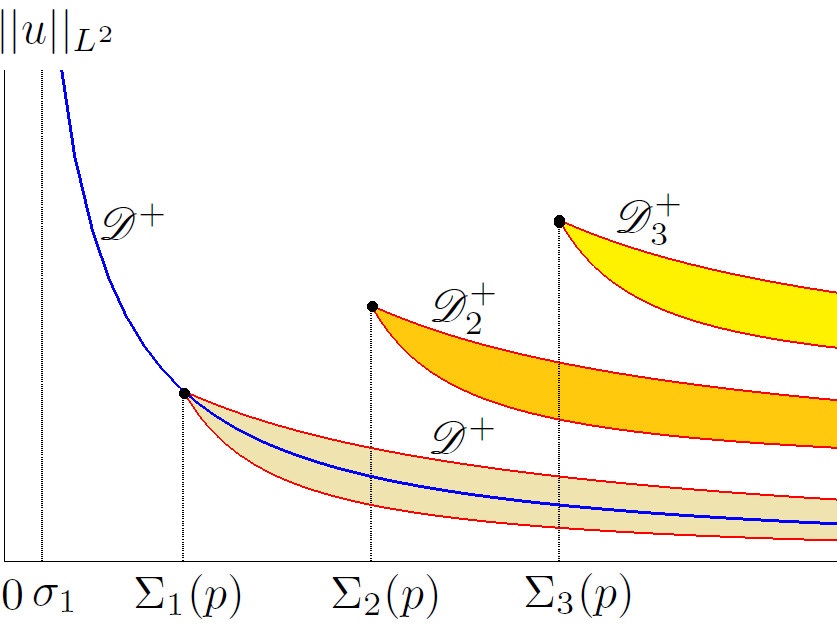}
    \caption{The components of positive solutions with $j\geq 1$ bumps, $\ms{D}_j^+$. We are denoting $\ms{D}_1^+\equiv \ms{D}^+$.     }
    \label{fig310b}
\end{figure}

\section{Final comments}\label{sub-7.end}
The discussion of Section 3 shows that even in the
simplest case of \eqref{1.5}, the problem \eqref{1.1}
can have multiple degenerate positive solutions
for sufficiently large $\l$.
Actually, in the context of the problem \eqref{1.5},
for each integer $j\geq 1$ and every $\l > \Sigma_j(p)$,
the problem \eqref{1.5} possesses a $\kappa$-simplex of positive solutions
with $\kappa$ bumps for all $\kappa \in\{1,...,j\}$, although the component $\ms{D}^+\equiv \ms{D}_1^+$ constructed in
Theorem \ref{th2.1} consists of solutions with a single bump.
\par
The fact that ${u}_{\Sigma_1(p)}$ vanishes with its first derivatives at the boundary of its domain,
allows us to extend it periodically to the whole real line. Moreover,
gluing together copies of ${u}_{\Sigma_1(p)}$ with the null function, we can also construct
positive subharmonic solutions of any order. Actually, the following result concerning \lq\lq chaotic solutions\rq\rq\,  holds.

\begin{theorem}
\label{th4.1}
Assume $L=2T_{H}(\Sigma_1(p))$.  Then, given any doubly infinite sequence of symbols
$\mathbf{s}=(s_i)_{i\in \mathbb{Z}}\in\{0,1\}^{\mathbb{Z}},$
there exists a solution $u=u_{\mathbf{s}}$
of \eqref{1.5} with the following properties:
\begin{itemize}
\item[{\rm (a)}] $u\geq 0$ and $\max u=u_{H}(K_1(p))$ if $u\neq 0$. Moreover,
$u(iL)=u'(iL)=0$ for all $i\in\mathbb{Z}$;
\item[{\rm (b)}] $u|_{((i-1)L,iL)}\equiv 0$ if $s_i=0$, while $u(x)> 0$
for all $x\in ((i-1)L,iL)$ if $s_i=1$;
\item[{\rm (c)}] $u=u_{\mathbf{s}}$ is a subharmonic solution of order $\kappa$ if the sequence $\mathbf{s}$ is $\kappa$-periodic (of minimal period $\kappa$).
\end{itemize}
\begin{proof}
Given any sequence of two symbols
$\mathbf{s}=(s_i)_{i\in \mathbb{Z}}\in\{0,1\}^{\mathbb{Z}}$,
a solution $u=u_{\mathbf{s}}$ can be constructed by gluing together
the null solutions in the intervals $[(i-1)L,iL]$ if $s_i=0$
and the solution
$$
  \hat{u}(x-(i-1)L-T_H(\Sigma_1(p)))=\hat{u}(x-(i-1)L-\tfrac{L}{2}),
\quad\text{if }\; (i-1)L\leq x \leq iL,
$$
if $s_i=1$. Then, it is straightforward to check that all the
assertions of the Theorem hold.
\end{proof}
\end{theorem}

The dynamical interpretation of the result is that in the phase plane $(u,u')$,
we have a trajectory which makes one loop following $\mathcal{O}^+$
in the time-interval $((i-1)L,iL)$ when $s_i=1,$
while the solution remains at the origin
in the time-interval $((i-1)L,iL)$, when $s_i=0$.
In any case, the solution is at
the origin at the times $x=iL$ for all $i\in\mathbb{Z}.$
\par
The same result holds if $L>2T_H(\Sigma_1(p))$,  but then, although
$$
  \max u|_{((i-1)L,iL)} = u_H(\Sigma_1(p))>0\quad \hbox{if}\;\; s_i=1,
$$
the solution $u=u_{\mathbf{s}}$ might vanish on
some points inside the interval $((i-1)L,iL)$.
\par
To add more complexity, one can consider chaotic dynamics on three symbols
$\{-1,0,1\}$ with a solution, which, in addition to
the above described dynamics, also traverses
a loop along the degenerate homoclinic $\mathcal{O}^-$ during the interval $((i-1)L,iL)$ when $s_i=-1$.
\par
The dynamics described in Theorem \ref{th4.1}
represents an example of chaos in the coin-tossing sense as described in \cite{KS}.
By that we mean for any given sequence
of ${\rm Heads}=0$ and ${\rm Tails}=1,$ a solution can be constructed with the
same associated coding (as in in $(b)$ of Theorem \ref{th4.1}).
Such behavior does not seem to fit into the typical definitions
of chaotic dynamics in the literature \cite{D,HK}.
Indeed most definitions in the literature
consider the iterates of a given single-valued continuous map that is
often a homeomorphism. However in our situation, the
Poincar\'{e} map associated with the solutions of system \eqref{3.3}
is not single-valued in the whole plane due to
the lack of uniqueness of the Cauchy problem associated
to \eqref{3.3}. A similar case of \lq\lq chaos\rq\rq\, for scalar first order
differential equations without uniqueness $u'=f(t,u)$ was
considered in \cite{OO}. Later, it was established in \cite{P} that, in spite of the
lack of uniqueness of the associated Cauchy problem, there is a way
to enter into the classical definitions of chaotic dynamics by
defining a suitable subsystem of the Bebutov flow
(for the basic theory about the Bebutov flows, see \cite{Se,Si}).
\par
Following a similar approach as in \cite{P}, we can
show that Theorem \ref{th4.1} provides indeed a true form of chaotic dynamics
if properly interpreted. To this end, we proceed as follows.
First of all we introduce the set
$$\mathscr{M}=\{u_{\mathbf{s}}:\,
\mathbf{s}=(s_i)_{i\in \mathbb{Z}}\in\{0,1\}^{\mathbb{Z}}\}$$
of all the solutions $u_{\mathbf{s}}$ of Theorem \ref{th4.1}.
The projection map
$$
   \Pi:\mathscr{M}\to \{0,1\}^{\mathbb{Z}}, \qquad \Pi(u_{\mathbf{s}}) = \mathbf{s},
$$
is a bijection. Actually, it is  a homeomorphism if we take on $\mathscr{M}$ the
topology of uniform convergence on compact sets and on
$\{0,1\}^{\mathbb{Z}}$ the standard product topology. Now, the
$L$-translation on the $x$-component $\Phi_{L}: \mathscr{M}\to \mathscr{M}$
defined by $\Phi_{L}: u(\cdot)\mapsto u(\cdot+L),$
is \textit{conjugate} with the shift automorphism (also called Bernoulli shift)
$$
    \vartheta: \{0,1\}^{\mathbb{Z}} \to \{0,1\}^{\mathbb{Z}}, \qquad
    \vartheta:(s_i)_{i\in {\mathbb Z}}\mapsto (s_{i+1})_{i\in {\mathbb Z}}.
$$
In other words,
$$
  \Pi\circ \Phi_L = \vartheta \circ \Pi.
$$
This means that $\Phi_{L}$ on $\mathscr{M}$ has
all the dynamical properties of $\vartheta$ on $\{0,1\}^{\mathbb{Z}}.$
Conjugation with the shift map on the set of sequence of $m\geq 2$
symbols is the best known form of ``chaos'' as it fulfills all the
typical features (such as transitivity, density of periodic points,
sensitive dependence on initial conditions, positive topological entropy)
usually associated with the concept of chaotic dynamics (see \cite{D,HK}
for the main definitions and concepts,
as well as \cite[Section I.5]{Sm} in the classical work of Smale).
\par
Symbolics dynamics results have also been obtained for more general settings
using gluing arguments from the calculus of variations. See, e.g., the survey paper
\cite{MR} and the references therein. For such gluing arguments, the heads and tails are replaced
by a collection of basic solutions of the equations.
Each formal concatenation of a collection of basic
solutions that makes sense geometrically then corresponds to
a nearby actual solution of the equation
that is obtained variationally. The nonuniqueness
in the current problem replaces the variational arguments
by simple direct concatenation, e.g., we get a $2$ bump solution by gluing $2$
one bump solutions that are
suitably separated.

\end{document}